\documentclass[a4paper,11pt]{elsarticle}
\usepackage{geometry}
\usepackage{stmaryrd}
\usepackage{amsmath,amssymb,amsthm,mathrsfs}
\usepackage{tabularx,epsfig,color,tabularx}
\usepackage{algorithm}
\usepackage{caption}
\usepackage{mathtools}
\usepackage{epsfig}
\usepackage{multirow}
\usepackage{tabularx}
\usepackage{booktabs}
\usepackage{subfigure}
\usepackage{makecell}

\allowdisplaybreaks[1]

\newcommand{\pt}{\partial}
\newcommand{\eps}{\varepsilon}

\newtheorem{theorem}{Theorem}[section]
\newtheorem{definition}{Definition}[section]
\newtheorem{lemma}{Lemma}[section]

\makeatletter 
\@addtoreset{equation}{section}
\@addtoreset{figure}{section}
\@addtoreset{table}{section}
\makeatother  
\renewcommand\theequation{{\thesection}%
	.{\arabic{equation}}}

\newcommand{\be}{\begin{equation}}
\newcommand{\ee}{\end{equation}}
\newcommand{\ba}{\begin{array}}
    \newcommand{\ea}{\end{array}}
\newcommand{\bea}{\begin{eqnarray}}
\newcommand{\eea}{\end{eqnarray}}
\newcommand{\beas}{\begin{eqnarray*}}
	\newcommand{\eeas}{\end{eqnarray*}}

\usepackage{verbatim}
\usepackage{tikz}
\definecolor{color1}{RGB}{249,255,195}
\definecolor{color2}{RGB}{0,176,80}
\usetikzlibrary{arrows.meta, arrows}
\usetikzlibrary{shapes,snakes}

\title{A sixth-order compact time-splitting Fourier pseudospectral method for the Dirac equation}
\author[1,2]{Weiguo Gao}
\ead{wggao@fudan.edu.cn}
\author[1]{Zhansi He}
\ead{23210180094@m.fudan.edu.cn}
\author[1]{Jia Yin}
\ead{jiayin@fudan.edu.cn}
		
\address[1]{School of Mathematical Sciences, Fudan University, Shanghai 200433, China}
\address[2]{School of Data Science, Fudan University, Shanghai 200433, China}

\begin{document}
    
    \begin{frontmatter}
		
    \begin{abstract}
    In this paper, we propose a novel sixth-order compact time-splitting scheme, denoted as $ S_{6\text{c}}$, for solving the Dirac equation in the absence of external magnetic potentials. This method is easy to implement, and it provides a substantial reduction in computational complexity compared to the existing sixth-order splitting schemes. By incorporating a time-ordering technique, we also extend $S_{6\text{c}}$ to address problems with time-dependent potentials. Comprehensive comparisons with various time-splitting methods show that $S_{6\text{c}}$ exhibits significant advantages in terms of both precision and efficiency. Moreover, numerical results indicate that $S_{6\text{c}}$ maintains the super-resolution property for the Dirac equation in the nonrelativistic regime in the absence of external magnetic potentials.
    \end{abstract}

	\end{frontmatter}

	\section{Introduction}\setcounter{equation}{0}
	Formulated by Paul Dirac in 1928, the Dirac equation stands as a foundational achievement in modern physics, as it integrates quantum mechanics with special relativity~\cite{Dirac}. This fundamental equation, which describes the behavior of relativistic spin-$1/2$ particles such as the electrons, extends applications across diverse scientific disciplines.~In condensed matter physics, the equation is instrumental in probing the structural and dynamical properties of two-dimensional materials such as graphene and graphite~\cite{FEF1, FEF2, Castro, Novo1, Novo2}. Moreover, it serves as a key theoretical tool for investigating relativistic effects in molecules subjected to intense fields like attosecond laser pulses~\cite{Boada,Fillion}. The Dirac equation is also crucial for studying foundational quantum phenomena, including the quantum Hall effect~\cite{Du, Goldman} and the physics governing topological insulators~\cite{Chen, Xia}.

	For simplicity, we only consider the Dirac equation in one and two-dimensional (1D and 2D) spaces in this paper. In 1D and 2D, the decoupled non-dimensionalized Dirac equation with electromagnetic potentials takes the form:
    \begin{equation}
	\label{LDirac1d2d}
	i\delta\partial_t\Phi = \left(-i\dfrac{\delta}{\eps}\sum_{j = 1}^{d}\sigma_j\partial_j + \dfrac{\nu}{\eps^2}
	\sigma_3 \right)\Phi + \left(V(t,\mathbf{x})I_2 - \sum_{j = 1}^{d}A_j(t,\mathbf{x})\sigma_j\right)\Phi, \quad \mathbf{x} \in\mathbb{R}^d,
	\end{equation}
    where $d \in \{1,2\}$ denotes the spatial dimension, and the initial condition is
    \begin{equation} \label{initial}
        \Phi(t=0,\mathbf{x}) = \Phi_0(\mathbf{x}),\quad \mathbf{x}\in \mathbb{R}^d.
    \end{equation}
    In the Dirac equation (\ref{LDirac1d2d}), $i = \sqrt{-1}$, $\Phi:=\Phi(t,\mathbf{x})=(\phi_1(t,\mathbf{x}),\phi_2(t,\mathbf{x})) \in \mathbb{C}^2$  is the two-component wave function, $t$ is the time, $\mathbf{x}:=(x_1,\cdots,x_d)^T $ is the spatial coordinate and $\partial_j$ represents $\partial_{x_j}$ for $j = 1, 2$. The dimensionless parameters $\delta,~\nu,~\eps$ satisfy $0 < \delta,~\nu,~\eps \le 1$ with specific physical meanings described in~\cite{Bao2019}. Besides, $V(t, \mathbf{x})\in\mathbb{R}$ is the electric potential and $\mathbf{A}(t, \mathbf{x}) = (A_1(t, \mathbf{x}), \cdots, A_d(t, \mathbf{x}))^T\in\mathbb{R}^d$ stands for the magnetic potential. $\sigma_j$ ($j=1, 2, 3$) in the equation are the Pauli matrices with the definition:
	\begin{equation}
	\label{Pauli}
	\sigma_1 = \begin{pmatrix} 0 & 1\\ 1 & 0 \end{pmatrix}, \quad 
	\sigma_2 = \begin{pmatrix} 0 & -i \\ i & 0\end{pmatrix}, \quad
	\sigma_3 = \begin{pmatrix} 1 & 0 \\ 0 & -1 \end{pmatrix}.
	\end{equation}
    Moreover, the Dirac equation (\ref{LDirac1d2d}) conserves the total mass as \cite{Bao2017}
    \begin{equation}
        \Vert \Phi(t,\cdot) \Vert ^2 := \int_{\mathbb{R}^d} |\Phi(t,\mathbf{x})|^2 d\mathbf{x} = \int_{\mathbb{R}^d} \sum_{j=1}^{2} |\phi_j(t,\mathbf{x})|^2 d \mathbf{x} \equiv \Vert \Phi(0,\cdot) \Vert ^2 = \Vert \Phi_0 \Vert ^2, t \ge 0.
    \end{equation}
    
	Extensive research has been devoted to both analytical and numerical solutions for the Dirac equation. On the analytical front, studies have focused on the existence and multiplicity of bound states and standing wave solutions, as detailed in~\cite{Das1993,Das1989,Esteban,Gesztesy,Gross,P} and references therein. On the numerical side, a variety of efficient and accurate methods have been developed and analyzed~\cite{Antoine,Bao2004}. Notable examples include the finite difference time domain (FDTD) method~\cite{Antoine2014,Braun,Hammer}, the time-splitting Fourier pseudospectral (TSFP) method~\cite{Bao2017,Braun,Fillion2012,Huang}, the exponential wave integrator Fourier pseudospectral (EWI-FP) method~\cite{Bao2017}, and the Gaussian beam method~\cite{Wu}. Furthermore, the Dirac equation has been widely studied in various physical regimes, such as the nonrelativistic regime~\cite{Bao2016,Bao20166,Bao2020,Cai} and the semiclassical regime~\cite{Ma}, through a combination of analytical and numerical approaches.

    In this paper, we focus on the time-splitting Fourier pseudospectral methods, which employ the Fourier discretization in space and time-splitting~\cite{McLachlan} for temporal integration. As noted in the literature~\cite{Bao2017,Bao2019}, existing splitting methods for solving the Dirac equation are limited to at most fourth-order accuracy. Among these methods, representative schemes include the first-order Lie-Trotter splitting ($S_1$)~\cite{Trotter}, the second-order Strang splitting ($S_2$)~\cite{Strang}, as well as several fourth-order variants, e.g., the Forest-Ruth splitting ($S_4$)~\cite{Forest,Suzuki,Yoshida}, the partitioned Runge-Kutta splitting ($S_{4\text{RK}}$)~\cite{Blanes,Geng} and the compact splitting ($S_{4\text{c}}$)~\cite{Chin1997}. While higher-order methods, such as those based on Suzuki fractal decomposition~\cite{Suzuki} or Yoshida symplectic integrators~\cite{Yoshida}, can  be applied straightforwardly, they require a large number of composition steps, resulting in prohibitive computational cost in practice.

    To achieve higher efficiency without sacrificing accuracy, in this paper we introduce a novel sixth-order compact time-splitting method ($S_{\text{6c}}$). The construction is inspired by the observation that, for the Dirac equation without magnetic potentials, the double commutator taken between the time-split operators vanishes~\cite{Bao2019,Yin2021}. This property leads to a compact scheme which is easy to implement in the absence of magnetic potentials. Moreover, by incorporating the time-ordering technique~\cite{Chin2002,Suzuki}, the proposed method can be extended to handle time-dependent potentials. 
    We will show through comprehensive numerical experiments that the new method is substantially more efficient than existing sixth-order compositions and delivers higher accuracy than conventional fourth-order schemes. Furthermore, it maintains the super-resolution property of the time-splitting methods for the Dirac equation without magnetic potentials in nonrelativistic regime.

    The rest of the paper is organized as follows. In Section \ref{sec2}, we first review the time-splitting methods for differential equations and then display the explicit construction procedures of the new sixth-order splitting scheme. Based on this construction, in Section \ref{sec3}, we propose a novel sixth-order compact time-splitting method for the Dirac equation with both time-independent and time-dependent electric potentials. Substantial numerical experiments and results are presented in Section \ref{sec4}, which demonstrate the advantages of the proposed $S_{6\text{c}}$ in terms of efficiency and accuracy. Finally, some concluding remarks are drawn in Section \ref{sec5}.

    \section{The sixth-order compact time-splitting method ($S_{6\text{c}}$)}\setcounter{equation}{0}\label{sec2}
    
    In this section, we review the time-splitting methods and introduce several properties which are necessary to construct the novel sixth-order splitting method for the Dirac equation without magnetic potentials. Based on these properties, we give the explicit $S_\text{6c}$ scheme in the last subsection.

    \subsection{Review of the time-splitting methods}
    The following discussion of the time-splitting methods is based on the differential equation:
    \begin{equation}
    \label{partial}
    \pt_tu(t, \mathbf{x}) = (T + W)u(t, \mathbf{x}),
    \end{equation}
    where $T$ and $W$ are two time-independent operators. With the initial condition 
    $u(0, \mathbf{x}) = u_0(\mathbf{x})$, we could represent the exact solution at time $\tau$ as
    \begin{equation}
    \label{exact}
     u(\tau, x) = e^{\tau(T + W)}u_0(x).
     \end{equation}
     When compared to the operators $e^{\tau T}$ and $e^{\tau W}$, $e^{\tau(T + W)}$ is not so easy to compute, we can apply the time-splitting technique, that is to approximate the operator $e^{\tau(T + W)}$ by
    \begin{equation}  \label{solution}
    e^{\tau(T + W)} \approx \Pi_{j = 1}^m e^{a_j\tau T}e^{b_j\tau W}, \quad m\in\mathbb{N}^*.
    \end{equation}
    For the convenience of notation, we denote the right hand side of the equation (\ref{solution}) as:
    \begin{align} 
        S(\tau) &:= \prod_{j=1}^{m} e^{a_j \tau T}e^{b_j \tau W}  \notag\\
        \label{order}   & = e^{\tau(T+W) + \mathcal{O}(\tau^{n+1})}\\  
        \label{coefficient}   & = e^{\sum_{k=1}^{n} \tau^k\gamma_k + \mathcal{O}(\tau^{n+1})} 
    \end{align}
    where (\ref{order}) means that $S(\tau)$ constructs an $n{\text{th}}$-order splitting method, and  (\ref{coefficient}) shows the asymptotic expansion of $S(\tau)$ with respect to $\tau$, where $\gamma_k$ is the coefficient of the $\tau^k$ term.

    \begin{definition}
        An algorithm $S(\tau)$ is \textbf{symmetric} or \textbf{self-adjoint} if 
        \begin{equation*}
            S(\tau)S(-\tau) = 1
        \end{equation*}
        for all $\tau$.
    \end{definition}

    \begin{theorem} \label{th1}
        If $S(\tau)$ is symmetric, then $\gamma_{2k} = 0 $ for all $k$, and $S(\tau)$ is a scheme with even order.
    \end{theorem}
    \noindent Theorem \ref{th1} plays a crucial role in constructing high-order splitting methods, and we refer to~\cite{Suzuki,Yoshida} for the proof. To construct higher-order time-splitting schemes, it is also necessary to recall the Baker-Campbell-Hausdorff (BCH) formula:
    \begin{definition}[the BCH formula]
        For any non-commutative operators $X$ and $Y$, the product of them, $\exp(X)\exp(Y)$, can be expressed as a single exponential operator
        \begin{equation*}
            \exp(X)\exp(Y) = \exp(Z),
        \end{equation*}
        where
        \begin{align}  \label{BCH}
            Z =& X + Y + \frac{1}{2}[X,Y] + \frac{1}{12}([X,[X,Y]]+[Y,[Y,X]]) \notag\\
              &+\frac{1}{24}[X,[Y,[Y,X]]]-\frac{1}{720}([Y,[Y,[Y,[Y,X]]]]+[X,[X,[X,[X,Y]]]]) \notag\\
              &+ \frac{1}{360}([Y,[X,[X,[X,Y]]]]+[X,[Y,[Y,[Y,X]]]]) \notag\\
              &+\frac{1}{120}([X,[X,[Y,[Y,X]]]]+[Y,[Y,[X,[X,Y]]]])+\cdots
        \end{align}
    \end{definition}
    \noindent Here we use the notation of the commutator $[X, Y] = X Y - Y X$. 
    
    Based on the above properties, Suzuki introduced the following \textbf{sixth-order Suzuki fractal splitting ($S_6^*$)} method~\cite{Suzuki} 
    \begin{align} \label{eq:S6f} 
    S^*_{6}(\tau) &=[S^*_{3}(p_3\tau)]^2S^*_{3}((1-4p_3)\tau)[S^*_{3}(p_3\tau)]^2, \notag\\
    S^*_{3}(\tau) &=[S_{2}(p_2\tau)]^2S_{2}((1-4p_2)\tau)[S_{2}(p_2\tau)]^2,
    \end{align}
    where $S_2(\tau)$ represents the second-order Strang splitting method~\cite{Strang}
    \begin{equation} 
     S_2(\tau)= e^{\frac{\tau}{2} W}e^{\tau T}e^{\frac{\tau}{2} W}.
    \end{equation}
    Here and in the following discussion we always suppose that the computation 
    of $e^{\tau W}$ is much easier than that of $e^{\tau T}$ without loss of generality. The coefficients in~\eqref{eq:S6f} are given by
    \begin{align*}
    p_2 &= 0.414490771794375737142354062860\cdots\\
    p_3 &= 0.373065827733272824775863041073\cdots.
    \end{align*}
    To reduce computational cost, Yoshida introduced  symmetric \textbf{sixth-order Yoshida symplectic splitting ($S_6$) } method~\cite{Yoshida} 
    \begin{equation} 
    S_6(\tau):=S_2(\omega_3\tau)S_2(\omega_2\tau) S_2(\omega_1 \tau)S_2(\omega_0\tau)S_2(\omega_1\tau) S_2(\omega_2\tau) S_2(\omega_3 \tau),
    \end{equation}
    with coefficients $\omega_1,\dots,\omega_3$ listed in Table \ref{tab1}, and $\omega_0$ is determined by
    \begin{equation}
        \omega_0 = 1 - 2(\omega_1 + \omega_2 + \omega_3).
    \end{equation}
    
    \begin{table}[t]
    \setlength{\tabcolsep}{10pt}
    \centering
    \begin{tabular}{cccc}
    \hline
         & solution A & solution B & solution C  \\
    \hline
      $\omega_1$  &   -1.17767998417887   &  
      -2.13228522200144  & 0.00152886228424922  \\
    
      $\omega_2$  &   0.235573213359357     &
      0.00426068187079180  & -2.14403531630539   \\
    
      $\omega_3$  &   0.784513610477560     & 
      1.43984816797678   & 1.44778256239930 \\
    \hline
    \end{tabular}
    \caption{The coefficients in the 6th-order Yoshida symplectic integrator.}
    \label{tab1}
    \end{table}

    \subsection{The construction of $S_{\text{6c}}$ for the Dirac equation}
    Since the computational costs of both $S_6^*$ and $S_6$ are prohibitive, we expect to design a more efficient sixth-order time-splitting scheme for the Dirac equation. In the following sections, we only consider the Dirac equation without magnetic potentials, i.e.,
    \begin{equation}  \label{Dirac1d2d}
        i\delta \partial_t \Phi = \left(-i\frac{\delta}{\eps}\sum_{j=1}^{d} \sigma_j \partial_j + \frac{\nu}{\eps^2}\sigma_3 \right) \Phi + V(t,\mathbf{x})I_2  \Phi, \quad \mathbf{x}\in \mathbb{R}^d,\quad d \in \{1,2\}.
    \end{equation}
    From the computation of the double commutators with both time-independent~\cite{Bao2019} and time-dependent~\cite{Yin2021} potentials, we can directly obtain the following lemma. 
    \begin{lemma}  \label{lemma1} 
    For the Dirac equation (\ref{Dirac1d2d}) in 1D and 2D with time-independent potentials, i.e., $V(t,\mathbf{x})\equiv V(\mathbf{x})$, by taking 
    \begin{equation}  \label{2.10}
        T = -\dfrac{1}{\eps}\sum_{j = 1}^{d}\sigma_j\partial_j - \dfrac{i\nu}{\delta\eps^2}\sigma_3 ,\quad W = -\frac{i}{\delta}V(\mathbf{x})I_2, \quad d \in \{1,2\},
    \end{equation}
    we have 
    \begin{equation}
        [W, [T, W ]] = 0.
    \end{equation}
    \end{lemma}

    To design higher order splitting schemes, we need to introduce several properties of the commutators. Here we denote the higher order commutators as $[X, Y, Z] := [X, [Y, Z] ]$. 
    
    For simplicity, the following discussion is only based on time-independent potentials. However, it can be easily extended to time-dependent cases with the help of the time-ordering technique \cite{Chin2002,Suzuki}. We first introduce several properties of the commutators.

    \begin{lemma} \label{lemma3}
    We have the following properties for the commutators: 
        \begin{enumerate}
            \item \textbf{Skew-Commutativity}: $[X,Y] = -[Y,X]$.     
            
            \item \textbf{Bilinearity}:
            \begin{align*}
                &[a_1X_1 + a_2X_2, Y] = a_1[X_1,Y] + a_2[X_2,Y] \\
                &[X, b_1Y_1 + b_2Y_2] = b_1[X,Y_1] + b_2[X,Y_2]
            \end{align*}
            for any $a_j, b_j~(j = 1,2)$ and operators $X_j, Y_j~(j =1,2)$.
            
            \item  \textbf{Jacobi Identity}:
            \begin{equation*}
                [X,Y,Z] + [Y,Z,X] + [Z,X,Y] = 0.
            \end{equation*}
            
            \item \textbf{The quadruple commutators satisfy}:
            \begin{equation*}
                [X,Y,Z,W] + [Y,Z,W,X] + [Z,W,X,Y] + [W,X,Y,Z] = [[X,Z],[Y,W]].
            \end{equation*}
        \end{enumerate}
    \end{lemma}

    \begin{proof}
        We only prove the last property. By definition:
        \begin{align*}
            [X,[Y,[Z,W]]] =& XYZW - XYWZ - XZWY + XWZY\\
            &- YZWX + YWZX + ZWYX - WZYX, \\
            [Y,[Z,[W,X]]] =& YZWX - YZXW - YWXZ + YXWZ\\
            &- ZWXY + ZXWY + WXZY - XWZY, \\
            [Z,[W,[X,Y]]] =& ZWXY - ZWYX - ZXYW + ZYXW\\
            &- WXYZ + WYXZ + XYWZ - YXWZ, \\
            [W,[X,[Y,Z]]] =& WXYZ - WXZY - WYZX + WZYX\\
            &- XYZW + XZYW + YZXW - ZYXW.
        \end{align*}
        
        Adding up all the above equations, we have
        \begin{align*}
            &[X,Y,Z,W] + [Y,Z,W,X] + [Z,W,X,Y] + [W,X,Y,Z]\\
            =& (XZ - ZX)(YW - WY) + (YW - WY)(ZX - XZ) \\
            =& [X,Z][Y,W] + [Y,W][Z,X] \\
            =& [[X,Z],[Y,W]],
        \end{align*}
        which completes the proof.
    \end{proof}

    \begin{lemma} \label{lemma4}
        If $[Y,X,Y] = 0$, then 
        \begin{enumerate}
        \item \label{(1)}$[Y,Y,X] = -[Y,X,Y] = 0$;
        \item \label{(2)}$[X,Y,Y,X] = [X,[Y,Y,X]] = 0$; 
        \item $[Y,Y,Y,Y,X] = [Y,Y,[Y,Y,X]] = 0$;
        \item $[X,Y,Y,Y,X] = [X,Y,[Y,Y,X]] = 0$;
        \item $[X,X,Y,Y,X] = [X,X,[Y,Y,X]] = 0$;
        \item \label{(6)}$[Y,Y,X,X,Y] = [X,Y,X,X,Y] = 0$. 
        \end{enumerate}
    \end{lemma}

    \begin{proof}
        The first five equations are straightforward to get, so here we only prove 6. By Lemma \ref{lemma3}, we have
        \begin{equation} 
            [Y,X,X,Y]+[X,X,Y,Y]+[X,Y,Y,X]+[Y,Y,X,X]=[[Y,X],[X,Y]] = 0.
        \end{equation}
        Notice that $[X,X,Y,Y]=[Y,Y,X,X]=0$ and $[Y,X,X,Y] = -[X,Y,Y,X] = 0$,
        thus
        \begin{align*}
            [Y,Y,X,X,Y] &= [Y,[Y,X,X,Y]] = 0,\\
            [X,Y,X,X,Y] &= [X,[Y,X,X,Y]] = 0.
        \end{align*}
        This completes the proof.
    \end{proof} 
    
    \begin{theorem}  \label{th2}
    Consider a symmetric decomposition of the form
    \begin{equation}   \label{2.15}
        S^{(n)}(\tau) := \cdots e^{\tau c_4 W} e^{\tau c_3 T} e^{\tau c_2 W}
        e^{\tau c_1 T} e^{\tau c_0 W}e^{\tau c_1 T} e^{\tau c_2 W}
        e^{\tau c_3 T} e^{\tau c_4 W}\cdots ,
    \end{equation}
    where $c_0,c_1,c_2,\cdots$ are any given coefficients and the operators $W$ and $T$ are given by (\ref{2.10}).
    Then $S^{(n)}(\tau)$ admits the expansion
    \begin{align*}
        S^{(n)}(\tau) = e^{
        \tau\bigl(a_1T+a_2W\bigr) 
        + \tau^3 a_3[T,W,T] + \tau^5\bigl(a_4[T,T,T,T,W]+a_5[W,T,T,T,W]\bigr)
        + \mathcal{O}(\tau^7)}, 
    \end{align*}
    in which the coefficients $a_1,\dots,a_5$ are polynomials in terms of $c_0,c_1,c_2,\cdots$.
    \end{theorem}

    \begin{proof}
        According to Theorem~\ref{th1}, any symmetric scheme of the form (\ref{2.15}) contains no even-order terms in its expansion. Thus $S^{(n)}(\tau)$ has the following form:
        \begin{align}    \label{gamma}
            S^{(n)}(\tau) 
            &= \exp\Bigl(\,\sum_{k=1}^{6} \tau^k \gamma_k + \mathcal{O}(\tau^{7})\Bigr), \qquad 
            \gamma_{2k}=0, \; \forall k\in\mathbb{N^*}, \notag\\
            &= \exp\Bigl(\tau \gamma_1 + \tau^3\gamma_3 + \tau^5\gamma_5 + \mathcal{O}(\tau^7)\Bigr).
        \end{align} 
        Expanding $S^{(n)}(\tau)$ repeatedly via the BCH formula shows that 
        $\gamma_1$ is a combination of $W$ and $T$,
        $\gamma_3$ is a combination of all double commutators in $W$ and $T$, and 
        $\gamma_5$ is a combination of all five-entry commutators in $W$ and $T$.
        By applying Lemma~\ref{lemma4}, we can remove the vanishing commutators, then
        \begin{align*}
            \gamma_1 &= a_1T + a_2W, \\ 
            \gamma_3 &= a_3[T,W,T], \\ 
            \gamma_5 &= a_4[T,T,T,T,W]+a_5[W,T,T,T,W],
        \end{align*}
        where the coefficients $a_1,\cdots,a_5$ are polynomials in terms of $c_0,c_1,c_2,\cdots$. Substituting these expressions into (\ref{gamma}) gives the desired results.
    \end{proof}
    
    In order to derive a sixth-order scheme, we should choose appropriate parameters in (\ref{2.15}), such that
    \begin{align}  \label{condition}
        \gamma_1 = W+T,\quad \gamma_3 = \gamma_5 =0,\quad i.e. ,\quad a_1 = a_2 = 1,\quad a_3 = a_4 = a_5 = 0.
    \end{align}
    It's worth mentioning that since the equation (\ref{condition}) has five conditions, we require to have five parameters $c_0,\cdots,c_4$ in the expression~\eqref{2.15}. Specifically, we suppose
    \begin{align} \label{2.18}
         S_{6c}(\tau) = e^{\tau c_4 W} e^{\tau c_3 T} e^{\tau c_2 W}e^{\tau c_1 T} e^{\tau c_0 W}e^{\tau c_1 T} e^{\tau c_2 W}e^{\tau c_3 T} e^{\tau c_4 W}.
    \end{align}
    To obtain (\ref{condition}), we rigorously expand (\ref{2.18}) and match the corresponding coefficients. This yields a system of algebraic equations for $c_0,\dots,c_4$, which we solve numerically.
    The following theorem is useful when expanding (\ref{2.18}):
    \begin{theorem} \label{th3}
        By repeated applications of the BCH formula~\eqref{BCH}, we have the following expansion:
    \begin{equation*}  
            \exp(X)\exp(Y)\exp(X) = \exp(U),
    \end{equation*}
    with
    \begin{align}   \label{BCH2}
        U =& 2X + Y + \frac{1}{6}[Y,[Y,X]]-\frac{1}{6}[X,[X,Y]] \notag\\
        &+ \frac{7}{360}[X,[X,[X,[X,Y]]]]- \frac{1}{360}[Y,[Y,[Y,[Y,X]]]] 
        + \frac{1}{90}[X,[Y,[Y,[Y,X]]]] \notag\\
        &+ \frac{1}{45}[Y,[X,[X,[X,Y]]]] - \frac{1}{60}[X,[X,[Y,[Y,X]]]] + \frac{1}{30}[Y,[Y,[X,[X,Y]]]] + \cdots
    \end{align}
    \end{theorem}
    
    Taking into account Lemmas \ref{lemma1}-\ref{lemma4} and Theorems \ref{th2}-\ref{th3}, we can expand the central part of (\ref{2.18}) as
    \begin{equation*}
        e^{\tau c_1 T}e^{\tau c_0 W}e^{\tau c_1 T} := e^{V_1}
    \end{equation*}
    where
    \begin{align}
        V_1 =& \tau(2c_1T + c_0W) - \tau^3\left(\frac{1}{6}[c_1T,c_1T,c_0W]\right) \notag\\
        &+\tau^5 \left(\frac{7}{360}[c_1T,c_1T,c_1T,c_1T,c_0W] +\frac{1}{45}[c_0W,c_1T,c_1T,c_1T,c_0W] \right)+\mathcal{O}(\tau^7) 
    \end{align}
    \noindent By applying Theorem \ref{th3} again to (\ref{2.18}) and simplifying (see details in ``Appendix A”), we have
    \begin{equation*}
    e^{\tau c_2 W}e^{\tau c_1 T} e^{\tau c_0 W}e^{\tau c_1 T} e^{\tau c_2 W} = e^{\tau c_2 W}e^{V_1}e^{\tau c_2 W} := e^{V_2}
    \end{equation*}
    where
    \begin{align} \label{V2}
        V_2 =& \tau\left(2c_1 T + (c_0 + 2c_2)W \right) +\tau^3 \left(\frac{1}{6} c_1^2(c_0-4c_2)[T,W,T] \right)\notag\\
        &+\tau^5(\frac{7}{360}c_0c_1^4-\frac{2}{45}c_1^4c_2)[T,T,T,T,W] \notag\\
        &+\tau^5(\frac{1}{45}c_0^2c_1^3-\frac{1}{18}c_0c_1^3c_2-\frac{1}{45}c_0c_1^3c_2+\frac{4}{45}c_1^3c_2^2)[W,T,T,T,W] + \mathcal{O}(\tau^7) 
    \end{align}
    For simplicity, let
    \begin{align}
     &e^{\tau c_3 T} e^{\tau c_2 W}e^{\tau c_1 T} e^{\tau c_0 W}e^{\tau c_1 T} e^{\tau c_2 W}e^{\tau c_3 T} = e^{\tau c_3 T}e^{V_2}e^{\tau c_3 T} =: e^{V_3} \label{eq:V3} \\
     &e^{\tau c_4 W} e^{\tau c_3 T} e^{\tau c_2 W}e^{\tau c_1 T} e^{\tau c_0 W}e^{\tau c_1 T} e^{\tau c_2 W}e^{\tau c_3 T} e^{\tau c_4 W} = e^{\tau c_4 W}e^{V_3}e^{\tau c_4 W} =: e^{V_4} \label{eq:V4}
    \end{align}
    By repeatedly applying Theorem \ref{th3} and simplifying (see details in ``Appendices B and C”), we can get the coefficients in $V_3$ and $V_4$ as shown in Tables \ref{tab2} and \ref{tab3}, respectively.

    \renewcommand{\arraystretch}{1.8}
    \begin{table}[t]
        \centering
        \resizebox{\textwidth}{!}{
        \begin{tabular}{m{0.8cm}<{\centering}cm{10cm}<{\centering}}
        \toprule[1.5pt]
           \multirow{2}{*}{$\tau$}  & $T$ & $2(c_1+c_3)$  \\
           \cline{2-3}
            & $W$ & $c_0+2c_2$  \\
        \midrule[0.8pt]
            $\tau^3$  & $[T,W,T]$  & $\frac{1}{6}c_1^2(c_0-4c_2)+\frac{1}{3}c_1c_3(c_0+2c_2)+\frac{1}{6}c_3^2(c_0+2c_2)$ \\
        \midrule[0.8pt]
        \multirow{2}{*}{$\tau^5$}  & $[T,T,T,T,W]$ & $c_1^4\left(\frac{7}{360}c_0-\frac{2}{45}c_2\right)+\frac{1}{18}c_1^3c_3(c_0-4c_2)+\frac{1}{36}c_1^2c_3^2(c_0-4c_2)+\frac{1}{45}c_1^3c_3(c_0+2c_2)
        +\frac{4}{45}c_1^2c_3^2(c_0+2c_2)+\frac{7}{90}c_1c_3^3(c_0+2c_2)+\frac{7}{360}c_3^4(c_0+2c_2)$ \\
        \cline{2-3}
        & $[W,T,T,T,W]$ & $c_1^3\left(\frac{1}{45}c_0^2-\frac{7}{90}c_0c_2+\frac{4}{45}c_2^2\right)+\frac{1}{18}c_1^2c_3(c_0+2c_2)(c_0-4c_2)+\frac{1}{90}c_1^2c_3(c_0+2c_2)^2+\frac{1}{45}c_3^3(c_0+2c_2^2)+\frac{1}{15}c_1c_3^2(c_0+2c_2)^2$ \\
        \bottomrule[1.5pt]
        \end{tabular}}
        \caption{Coefficients for $V_3$ in~\eqref{eq:V3}.}
        \label{tab2}
    \end{table}

    To achieve (\ref{condition}), we require
    \begin{equation}  \label{equation}
    \left \{ 
    \small{
        \begin{aligned}
            a_1 =& 2(c_1+c_3) = 1 \\
            a_2 =& c_0+2c_2+2c_4 = 1 \\
            a_3 =& \frac{1}{6}c_1^2(c_0-4c_2)+\frac{1}{3}c_1c_3(c_0+2c_2)+\frac{1}{6}c_3^2(c_0+2c_2)-\frac{2}{3}(c_1+c_3)^2c_4 = 0 \\
            a_4 =& c_1^4(\frac{7}{360}c_0-\frac{2}{45}c_2)+\frac{1}{18}c_1^3c_3(c_0-4c_2)+\frac{1}{36}c_1^2c_3^2(c_0-4c_2) \\
            &+\frac{1}{45}c_1^3c_3(c_0+2c_2) 
            +\frac{4}{45}c_1^2c_3^2(c_0+2c_2)+\frac{7}{90}c_1c_3^3(c_0+2c_2) \\
            &+\frac{7}{360}c_3^4(c_0+2c_2)-\frac{2}{45}(c_1+c_3)^4c_4 = 0\\
            a_5 =& c_1^3(\frac{1}{45}c_0^2-\frac{7}{90}c_0c_2+\frac{4}{45}c_2^2)+\frac{1}{18}c_1^2c_3(c_0+2c_2)(c_0-4c_2) \\
            &+\frac{1}{90}c_1^2c_3(c_0+2c_2)^2 
            + \frac{1}{45}c_3^3(c_0+2c_2^2)+\frac{1}{15}c_1c_3^2(c_0+2c_2)^2 \\
            &- \frac{1}{45}(c_0+2c_2)(c_1+c_3)^3c_4+\frac{4}{45}(c_1+c_3)^3c_4^2 -\frac{1}{18}c_1^2(c_0-4c_2)c_4(c_1+c_3) \\
            &-\frac{1}{9}c_1c_3(c_1+c_3)(c_0+2c_2)c_4-\frac{1}{18}c_3^2(c_1+c_3)(c_0+2c_2)c_4 = 0
        \end{aligned}}
    \right.
    \end{equation}
    Employing MATLAB to solve (\ref{equation}), we get one possible solution as given in Table \ref{tab4}.

    \begin{table}[t]
        \centering
        \resizebox{\textwidth}{!}{
        \begin{tabular}{m{0.8cm}<{\centering}cm{10cm}<{\centering}}
        \toprule[1.5pt]
           \multirow{2}{*}{$\tau$}  & $T$ & $2(c_1+c_3)$  \\
           \cline{2-3}
            & $W$ & $c_0+2c_2+2c_4$  \\
        \midrule[0.8pt]
            $\tau^3$  & $[T,W,T]$  & $\frac{1}{6}c_1^2(c_0-4c_2)+\frac{1}{3}c_1c_3(c_0+2c_2)+\frac{1}{6}c_3^2(c_0+2c_2)-\frac{2}{3}(c_1+c_3)^2c_4$ \\
        \midrule[0.8pt]
        \multirow{2}{*}{$\tau^5$}  & $[T,T,T,T,W]$ & $c_1^4(\frac{7}{360}c_0-\frac{2}{45}c_2)+\frac{1}{18}c_1^3c_3(c_0-4c_2)+\frac{1}{36}c_1^2c_3^2(c_0-4c_2)+\frac{1}{45}c_1^3c_3(c_0+2c_2)
        +\frac{4}{45}c_1^2c_3^2(c_0+2c_2)+\frac{7}{90}c_1c_3^3(c_0+2c_2)+\frac{7}{360}c_3^4(c_0+2c_2)-\frac{2}{45}(c_1+c_3)^4c_4$  \\
        \cline{2-3}
        & $[W,T,T,T,W]$ & $c_1^3(\frac{1}{45}c_0^2-\frac{7}{90}c_0c_2+\frac{4}{45}c_2^2)+\frac{1}{18}c_1^2c_3(c_0+2c_2)(c_0-4c_2)+\frac{1}{90}c_1^2c_3(c_0+2c_2)^2+\frac{1}{45}c_3^3(c_0+2c_2^2)+\frac{1}{15}c_1c_3^2(c_0+2c_2)^2-\frac{1}{45}(c_0+2c_2)(c_1+c_3)^3c_4+\frac{4}{45}(c_1+c_3)^3c_4^2-\frac{1}{18}c_1^2(c_0-4c_2)c_4(c_1+c_3)-\frac{1}{9}c_1c_3(c_1+c_3)(c_0+2c_2)c_4-\frac{1}{18}c_3^2(c_1+c_3)(c_0+2c_2)c_4$  \\
        \bottomrule[1.5pt]
        \end{tabular}}
        \caption{Coefficients for $V_4$ in~\eqref{eq:V4}.}
        \label{tab3}
    \end{table}

    \renewcommand{\arraystretch}{1.2}
    \begin{table}[t]
        \setlength{\tabcolsep}{20pt}
        \centering
        \setlength{\tabcolsep}{25pt}
        \begin{tabular}{cc}
        \toprule
               & Solution \\
        \midrule
         $c_0$ & 0.56752783701702083198289862647776935943946193310303 \\
         $c_1$ & -0.57985242638243088245699127515351743868902398276509 \\
         $c_2$ & -0.14371472730265404347711307162339049261607229413988 \\
         $c_3$ & 1.0798524263824308824569912751535174386890239827651 \\
         $c_4$ & 0.35995080879414362748566375838450581289634132758836 \\
        \bottomrule
        \end{tabular}
        \caption{One possible solution of the coefficients for the sixth-order compact splitting method.}
        \label{tab4}
    \end{table}
    
    Thus, we propose the \textbf{sixth-order compact splitting} ($S_{6\text{c}}$) method for the Dirac equation (\ref{Dirac1d2d}):
    \begin{equation}  \label{S_6c}
        S_{6c}(\tau) = e^{\tau c_4 W} e^{\tau c_3 T} e^{\tau c_2 W}e^{\tau c_1 T} e^{\tau c_0 W}e^{\tau c_1 T} e^{\tau c_2 W}e^{\tau c_3 T} e^{\tau c_4 W}
    \end{equation}
    with coefficients $c_0,\cdots,c_4$ from Table \ref{tab4}. A major advantage of $S_\text{6c}$ is the requirement for merely four integrations of the operator $T$ and five integrations of the operator $W$. In addition, there is no need to compute any commutators. Table \ref{tab5} compares the numbers of operators $T$ and $W$ implemented in different sixth-order time-splitting methods. From the table, we can make the following observations: (i) the computational cost of $S_6$ is over twice compared to $S_{6\text{c}}$; (ii)  the computational cost of $S_6^*$ is over five times compared to $S_{6\text{c}}$; (iii) among the three sixth-order splitting methods, $S_{6\text{c}}$ is the most efficient and $S_6^*$ is the most expensive.
    
    \begin{table}[t]
        \centering
        \setlength{\tabcolsep}{40pt}
        \begin{tabular}{cccc}
        \toprule
            & $S_6^*$ & $S_6$ & $S_{6\text{c}}$    \\
        \midrule
        $T$  & 25 & 9 & 4 \\
        \midrule
        $W$  & 26 & 10 & 5 \\
        \bottomrule
        \end{tabular}
        \caption{The numbers of operators $T$ and $W$ to be implemented in different sixth-order time-splitting methods.}
        \label{tab5}
    \end{table}


\section{The full-discretization of $S_{\text{6c}}$}\label{sec3}

    Applying a Fourier pseudospectral discretization in space, this section develops the full discretization of $S_\text{6c}$ from Section \ref{sec2} for the Dirac equation without magnetic potentials. We begin by addressing time-independent potentials (i.e. $V(t,\mathbf{x})\equiv V(\mathbf{x})$), and thereafter extend the approach to time-dependent cases through the time-ordering technique. For simplicity, we use the 1D Dirac equation (\ref{Dirac1d2d}) as an example, and we remark that the method can be easily extended to 2D cases.

    \subsection{$S_\text{6c}$ with time-independent potentials}
    We first consider time independent potentials, i.e. $V(t,\mathbf{x})\equiv V(\mathbf{x})$ in (\ref{Dirac1d2d}). Similar to most works in the literature on the analysis and computation of the Dirac equation~\cite{Bao2016,Bao2017,Bao20166,Bao2004,Bao2019}, we truncate the spatial domain to a bounded interval $\Omega = (a, b)$, and apply periodic boundary conditions for practical computation. The truncated interval is chosen sufficiently large to ensure negligible truncation error. Under these conditions, the 1D Dirac equation without magnetic potentials is as follows:
    \begin{equation}   \label{Dirac}
    \left \{
        \begin{aligned}
            & i\delta \partial_t \Phi = \left(-i\frac{\delta}{\eps} \sigma_1 \partial_x + \frac{\nu}{\eps^2}\sigma_3 \right) \Phi + V(x)I_2  \Phi, \quad x\in \Omega, \quad t \ge 0, \\
            & \Phi(t,a) = \Phi(t,b),\quad \partial_x\Phi(t,a) = \partial_x\Phi(t,b),\quad t \ge 0,\\
            & \Phi(0,x) = \Phi_0(x), \quad a \le x \le b,
        \end{aligned}
    \right.
    \end{equation}
    where $\Phi:=\Phi(t,x),\Phi_0(a)=\Phi_0(b)$ and $\Phi_0'(a)=\Phi_0'(b)$.
    
    With a time step $\tau > 0$, we denote $t_n = n\tau$ ($\forall n \in\mathbb{N}$) and let $\Phi^n(x)$ represent the numerical approximation to $\Phi(t_n, x)$. We reformulate the Dirac equation (\ref{Dirac}) as
    \begin{equation}  \label{WT}
    \partial_t\Phi =\left( -\frac{1}{\eps} \sigma_1 \partial_1 - \frac{i\nu}{\delta \eps^2}\sigma_3\right)\Phi -\frac{i}{\delta} V(x)I_2\Phi = (T+W)\Phi,
    \end{equation}
    where the operators $T$ and $W$ are defined in~\eqref{2.10} with $d=1$. Then \(S_{6\text{c}}\) (\ref{S_6c}) can be applied to perform temporal integration over the interval \([t_n, t_{n+1}]\):
    \begin{align}  \label{3.3}
        \Phi^{n+1}(x)=S_{6c}(\tau)\Phi^n(x)=e^{\tau c_4 W} &e^{\tau c_3 T} e^{\tau c_2 W}e^{\tau c_1 T} e^{\tau c_0 W}e^{\tau c_1 T} e^{\tau c_2 W}e^{\tau c_3 T} e^{\tau c_4 W}\Phi^n(x),\notag \\
        &a\le x\le b,\quad n \ge 0,
    \end{align}
    where the values of the coefficients $c_0, \dots, c_4$ are listed in Table \ref{tab4}. To compute $e^{\tau c_k T}$ ($k = 1,3$), we first discretize the functions spatially via the Fourier pseudospectral method, and then perform temporal integration in the phase space (or the Fourier space) \cite{Bao2017,Bao2004}. On the other hand, the diagonal structure of the operator $W$ enables efficient computation of $e^{\tau c_k W}$ ($k = 0,2,4$).
    
    We define the mesh size $h := \triangle x = \frac{b 
        - a}{M}$, with $M$ being a positive even integer, then the grid points can be given as
    \begin{equation}
    x_j := a + jh, \quad j = 0, 1, \cdots , M.
    \end{equation}
    The exact solution is approximated by $\Phi_j^n \approx \Phi(t_n, x_j)$, 
    and the initial and boundary conditions can be discretized as
    \begin{equation}
    \begin{aligned}
    &\Phi_j^0 = \Phi_0(x_j), \quad j = 0, \cdots, M,\\
    &\Phi_0^n = \Phi_M^n, \quad \Phi_{-1}^n = \Phi_{M - 1}^n, \quad n = 0, 1,\cdots.
    \end{aligned}
    \end{equation}
    Let $X_M = \{U = (U_0,U_1,\cdots,U_M)^T|U_j \in \mathbb{C}^2,j=0,1,\cdots,M, U_0 = U_M\}$, then for all $U \in X_M$, we denote its Fourier representation as
    \begin{equation}
        U_j = \sum_{l=-M/2}^{M/2-1} \widetilde{U}_l e^{i \mu_l(x_j-a)} 
            = \sum_{l=-M/2}^{M/2-1} \widetilde{U}_l e^{2ijl\pi/M}, \quad j = 0,1,\cdots,M,
    \end{equation}
    where $\mu_l$ and $\widetilde{U}_l \in \mathbb{C}^2$ are:
    \begin{equation}
        \mu_l = \frac{2l\pi}{b-a}, \quad \widetilde{U}_l = \frac{1}{M} \sum_{j=0}^{M-1} U_j e^{-2ijl\pi/M}, \quad l = -\frac{M}{2},\cdots,\frac{M}{2}-1.
    \end{equation}
    Then the full-discretization of $S_\text{6c}$ for the Dirac equation (\ref{Dirac}) with $j = 0, 1, \cdots, M$ is given as
	\begin{align} 
        \Phi_j^{(1)} &= e^{\tau c_4 W(x_j)} \Phi_j^n = e^{-ic_4\tau \Lambda_1(x_j)} \Phi_j^n, \notag\\
        \Phi_j^{(2)} &= \sum_{l=-M/2}^{M/2-1} e^{c_3 \tau \Gamma_l} (\widetilde{\Phi^{(1)}})_l e^{i \mu_l(x_j-a)} = \sum_{l=-M/2}^{M/2-1} Q_le^{-ic_3\tau D_l} Q_l^*(\widetilde{\Phi^{(1)}})_l e^{2ijl\pi/M}, \notag\\
        \Phi_j^{(3)} &= e^{\tau c_2 W(x_j)} \Phi_j^{(2)} = e^{-ic_2\tau \Lambda_1(x_j)} \Phi_j^{(2)}, \notag\\
        \Phi_j^{(4)} &= \sum_{l=-M/2}^{M/2-1} e^{c_1 \tau \Gamma_l} (\widetilde{\Phi^{(3)}})_l e^{i \mu_l(x_j-a)} = \sum_{l=-M/2}^{M/2-1} Q_le^{-ic_1\tau D_l} Q_l^*(\widetilde{\Phi^{(3)}})_l e^{2ijl\pi/M}, \notag\\
        \Phi_j^{(5)} &= e^{\tau c_0 W(x_j)} \Phi_j^{(4)} = e^{-ic_0\tau \Lambda_1(x_j)} \Phi_j^{(4)}, \notag\\
        \Phi_j^{(6)} &= \sum_{l=-M/2}^{M/2-1} e^{c_1 \tau \Gamma_l} (\widetilde{\Phi^{(5)}})_l e^{i \mu_l(x_j-a)} = \sum_{l=-M/2}^{M/2-1} Q_le^{-ic_1\tau D_l} Q_l^*(\widetilde{\Phi^{(5)}})_l e^{2ijl\pi/M}, \notag\\
        \Phi_j^{(7)} &= e^{\tau c_2 W(x_j)} \Phi_j^{(6)} = e^{-ic_2\tau \Lambda_1(x_j)} \Phi_j^{(6)},\notag\ \\
        \Phi_j^{(8)} &= \sum_{l=-M/2}^{M/2-1} e^{c_3 \tau \Gamma_l} (\widetilde{\Phi^{(7)}})_l e^{i \mu_l(x_j-a)} = \sum_{l=-M/2}^{M/2-1} Q_le^{-ic_3\tau D_l} Q_l^*(\widetilde{\Phi^{(7)}})_l e^{2ijl\pi/M}, \notag\\
        \Phi_j^{n+1} &= e^{\tau c_4 W(x_j)} \Phi_j^{(8)} = e^{-ic_4\tau \Lambda_1(x_j)} \Phi_j^{(8)},
    \end{align}
    where
    \begin{align*}
        &W(x_j) := -\frac{i}{\delta}V(x_j)I_2 = -i\Lambda_1(x_j),\qquad j=0,1,\cdots,M,  \\
        &\Gamma_l = -\frac{i\mu_l}{\eps}\sigma_1-\frac{i\nu}{\delta\eps^2}\sigma_3 = - iQ_lD_lQ_l^*, \quad l = -\frac{M}{2},\cdots,\frac{M}{2}-1, 
    \end{align*}
    with 
    \begin{align*}
         &D_l = \text{diag}(\frac{1}{\delta \eps^2}\sqrt{\nu^2 + \delta^2 \eps^2 \mu_l^2},-\frac{1}{\delta \eps^2}\sqrt{\nu^2 + \delta^2 \eps^2 \mu_l^2}), \\
         &\Lambda_1 (x_j) = \text{diag}(\frac{1}{\delta}V(x_j),\frac{1}{\delta}V(x_j))
    \end{align*}
    and
    \begin{equation*}
        Q_l = \frac{1}{\sqrt{2\eta_l(\eta_l+\nu)}}
        \begin{pmatrix}
            \eta_l+\nu & -\delta\eps\mu_l \\
            \delta\eps\mu_l & \eta_l+\nu \\
        \end{pmatrix},
        \qquad l = -\frac{M}{2},\cdots,\frac{M}{2}-1,
    \end{equation*}
    with $\eta_l=\sqrt{\nu^2 + \delta^2 \eps^2 \mu_l^2}$.
    
    \subsection{$S_\text{6c}$ with time-dependent potentials} \label{sec:tdS6c}
    For time-dependent potentials, we can apply the time-ordering technique \cite{Chin2002,Suzuki} to extend the $S_\text{6c}$ method. Specifically, we consider the following model equation in 1D and 2D ($d=1, 2$):
    \begin{equation} \label{dependent}
    \partial_t u(t,\mathbf{x}) = \big(T + W(t)\big) u(t, \mathbf{x}), \quad t>t_0, \quad \mathbf{x}\in \mathbb{R}^d,
    \end{equation}
    with the initial condition
    \begin{equation}
    u(t_0,\mathbf{x}) = u_0(\mathbf{x}), \quad x\in \mathbb{R}^d,
    \end{equation}
    where $t_0$ denotes the initial time, $T$ is a time-independent operator, and $W(t)$ is time-dependent.

    Define a forward time derivative operator \cite{Chin2002}
    $\mathcal{D}:=\frac{\overset{\leftarrow}{\partial}}{\partial t}$. Its action on a time-dependent function $f(t)$ to its left is given by
    \begin{equation}
    f(t)\mathcal{D}=\lim_{\tau \rightarrow 0} \frac{f(t+\tau)-f(t)}{\tau}.
    \end{equation}
    It is straightforward to verify that
    \begin{equation} \label{3.28}
    F(t)e^{\tau \mathcal{D}}G(t) = F(t + \tau)G(t), \quad t>0,
    \end{equation}
    where $F(\cdot)$ and $G(\cdot)$ are two arbitrary time-dependent functions. Define $\widetilde{T}=T+\mathcal{D}$, the time-ordering technique allows us to express $u(t+\tau)$ as \cite{Yin2021}
    \begin{align} \label{3.29}
    u(t+\tau) = \exp [\tau(\widetilde{T}+W(t))]u(t), \quad t>0,
    \end{align}
    which represents the exact exponential solution for (\ref{dependent}). 

    Based on the time-ordering technique, we can now apply the splitting methods to the Dirac equation (\ref{Dirac1d2d}) with time-dependent potentials. Let $T$ and $W(t)$ be defined as in (\ref{2.12}), then we can get the following lemma from direct calculation.

    \begin{lemma} \label{lemma2}
        For the Dirac equation (\ref{Dirac1d2d}) in 1D and 2D with time-dependent potentials, by taking 
        \begin{equation}   \label{2.12}
            T = -\frac{1}{\eps}\sum_{j=1}^d \alpha_j \partial_j - \frac{i\nu}{\delta\eps^2}\sigma_3,\quad \widetilde{T} = T+\mathcal{D}, \quad  W(t) = -\frac{i}{\delta}V(t,\mathbf{x})I_2,\quad d \in \{1,2\},
        \end{equation}
        we have  
        \begin{equation}
            [W(t), [\widetilde{T}, W(t) ]] = 0.
        \end{equation}
    \end{lemma}
    
    This lemma implies that the previously constructed $S_{6\text{c}}$ method remains valid for the Dirac equation (\ref{Dirac1d2d}) with time-dependent potentials. Specifically, for a time step $\tau > 0$, applying $S_{6\text{c}}$ to the exact solution (\ref{3.29}) with $u(t,\mathbf{x}):= \Phi(t,\mathbf{x})$ yields
    \begin{align}   \label{3.31}
        \Phi(t+\tau) \approx & S_{6c}\Phi(t) \notag\\
        :=& e^{\tau c_4 W(t)}e^{\tau c_3 \widetilde{T}} e^{\tau c_2 W(t)}e^{\tau c_1 \widetilde{T}} e^{\tau c_0 W(t)}e^{\tau c_1 \widetilde{T}} e^{\tau c_2 W(t)}e^{\tau c_3 \widetilde{T}} e^{\tau c_4 W(t)}\Phi(t),\notag \\
        =& e^{\tau c_4 W(t + 2(c_1 + c_3)\tau)} e^{\tau c_3 T} e^{\tau c_2 W(t + (2c_1 + c_3) \tau)}e^{\tau c_1 T}...\notag\\
        &~~ e^{\tau c_0 W(t + (c_1 + c_3)\tau)}e^{\tau c_1 T} e^{\tau c_2 W(t + c_3 \tau)}e^{\tau c_3 T} e^{\tau c_4 W(t)}\Phi(t),\notag \\
        =& e^{\tau c_4 W(t + \tau)} e^{\tau c_3 T} e^{\tau c_2 W(t + (1 - c_3) \tau)}e^{\tau c_1 T}... \notag\\
        &~~ e^{\tau c_0 W(t + \frac{1}{2}\tau)}e^{\tau c_1 T} e^{\tau c_2 W(t + c_3 \tau)}e^{\tau c_3 T} e^{\tau c_4 W(t)}\Phi(t),
    \end{align}
    where the coefficients $c_0,\dots,c_4$ are given in Table \ref{tab4}, the second equality employs the result from (\ref{3.28}), and the third equality relies on the identity $2(c_1+c_3)=1$. 
    
    Denote $t_n = n\tau$ ($\forall n \in\mathbb{N}$), and denote $\Phi^n(x)$ as the numerical approximation to $\Phi(t_n,x)$, applying $S_{6\text{c}}$ over the time interval $[t_n, t_{n+1}]$ gives:
    \begin{align}   \label{3.17}
    \Phi^{n+1}(x) =& e^{\tau c_4 W(t_{n+1})} e^{\tau c_3 T} e^{\tau c_2 W(t_{n+1}-c_3\tau)} e^{\tau c_1 T}... \notag \\
    &~~e^{\tau c_0 W(t_n+\frac{1}{2}\tau)} e^{\tau c_1 T} e^{\tau c_2 W(t_n + c_3 \tau)} e^{\tau c_3 T} e^{\tau c_4 W(t_n)} \Phi^n(x),
    \end{align}
    with the initial condition
    \begin{equation}
    \Phi^0(x) := \Phi_0(x), \quad x \in \mathbb{R}.
    \end{equation}
    The computation of $e^{\tau c_k T}$ ($k = 1,3$) and $e^{\tau c_k W}$ ($k = 0, 2, 4$) with full-discretization follows the same procedure as with the time-independent potential: for $e^{\tau c_k T}$ ($k = 1,3$), we apply spatial discretization via the Fourier pseudospectral method and then integrate in time in the phase space (or the Fourier space) \cite{Bao2017,Bao2004}; for $e^{\tau c_k W}$ ($k = 0, 2, 4$), the diagonal structure of the operator $W(t)$ enables efficient computation.

    \subsection{Mass conservation}
    
	The $S_\text{6c}$ method with time-independent  potentials conserves the discretized mass, as shown in the following lemma.
    \begin{lemma}
        For any $\tau>0$, the $S_{\text{6c}}$ method (\ref{3.3}) with time-independent potentials in 1D conserves the mass in the discretized level, i.e., 
        \begin{equation} \label{3.19}
        \Vert \Phi^{n+1} \Vert^2_{l^2} := h \sum_{j=0}^{M-1}|\Phi_j^{n+1}|^2 \equiv h \sum_{j=0}^{M-1}|\Phi_j^{0}|^2 = h \sum_{j=0}^{M-1}|\Phi_0(x_j)|^2 = \Vert \Phi_{0} \Vert^2_{l^2} , \quad n \ge 0.
    \end{equation}
    \begin{proof}
       Noticing $W(x_j)^* = - W(x_j)$, and thus $\left(e^{\tau c_4 W(x_j)}\right)^*e^{\tau c_4 W(x_j)} = I_2$, from (\ref{3.3}) and summing for $j = 0,1,\dots,M-1$ , we have
        \begin{align}
             \Vert \Phi^{n+1} \Vert^2_{l^2} &= h \sum_{j=0}^{M-1}|\Phi_j^{n+1}|^2 = h \sum_{j=0}^{M-1} |e^{\tau c_4 W(x_j)}\Phi_j^{(8)}|^2 \notag\\
             &= h \sum_{j=0}^{M-1} \left( \Phi_j^{(8)} \right)^*\left(e^{\tau c_4 W(x_j)}\right)^*e^{\tau c_4 W(x_j)}\Phi_j^{(8)} \notag\\
             &= h \sum_{j=0}^{M-1} \left( \Phi_j^{(8)} \right)^* I_2 \Phi_j^{(8)} = h \sum_{j=0}^{M-1}|\Phi_j^{(8)}|^2 = \Vert \Phi^{(8)} \Vert^2_{l^2}, \quad n \ge 0.
        \end{align}
        Similarly, we have
        \begin{align} \label{3.21}
             &\Vert \Phi^{(7)} \Vert^2_{l^2} = \Vert \Phi^{(6)} \Vert^2_{l^2}, \quad \Vert \Phi^{(5)} \Vert^2_{l^2} = \Vert \Phi^{(4)} \Vert^2_{l^2}, \notag\\
             &\Vert \Phi^{(3)} \Vert^2_{l^2} = \Vert \Phi^{(2)} \Vert^2_{l^2},\quad \Vert \Phi^{(1)} \Vert^2_{l^2} = \Vert \Phi^{n} \Vert^2_{l^2}, \quad n \ge 0.
        \end{align}
        Moreover, by using the Parseval’s identity and noticing $\Gamma_l^* = -\Gamma_l$, thus $\left(e^{\tau \Gamma_l}\right)^*e^{\tau \Gamma_l} = I_2$, we have
        \begin{align} \label{3.22}
            &\Vert \Phi^{(8)} \Vert^2_{l^2} = \Vert \Phi^{(7)} \Vert^2_{l^2}, \quad \Vert \Phi^{(6)} \Vert^2_{l^2} = \Vert \Phi^{(5)} \Vert^2_{l^2},\notag\\
            &\Vert \Phi^{(4)} \Vert^2_{l^2} = \Vert \Phi^{(3)} \Vert^2_{l^2},\quad \Vert \Phi^{(2)} \Vert^2_{l^2} = \Vert \Phi^{(1)} \Vert^2_{l^2}.
        \end{align}
        Combining (\ref{3.21}) and (\ref{3.22}), we have
        \begin{align}
            \Vert \Phi^{n+1} \Vert^2_{l^2} &= \Vert \Phi^{(8)} \Vert^2_{l^2} = \Vert \Phi^{(7)} \Vert^2_{l^2} = \Vert \Phi^{(6)} \Vert^2_{l^2} = \Vert \Phi^{(5)} \Vert^2_{l^2} \notag\\
            &=\Vert \Phi^{(4)} \Vert^2_{l^2} = \Vert \Phi^{(3)} \Vert^2_{l^2} = \Vert \Phi^{(2)} \Vert^2_{l^2} = \Vert \Phi^{(1)} \Vert^2_{l^2} = \Vert \Phi^{n} \Vert^2_{l^2}.
        \end{align}
        Using the mathematical induction, we get the mass conservation (\ref{3.19}).
    \end{proof}
    \end{lemma}

\section{Numerical results} \label{sec4}
\setcounter{equation}{0}
\setcounter{table}{0}
\setcounter{figure}{0}

    In this section, we compare the sixth-order compact time-splitting ($S_{6\text{c}}$)  method with other time-splitting methods including  the fourth-order Forest-Ruth time-splitting ($S_4$) method, the fourth-order partitioned Runge-Kutta time-splitting ($S_\text{4RK}$) method, the fourth-order compact time-splitting ($S_{4\text{c}}$) method and the sixth-order Yoshida symplectic time-splitting ($S_6$) method in terms of accuracy and efficiency (here we ignore $S_6^*$ due to too many exponential operator evaluations). In these numerical simulations, as a common practice, we would truncate the whole space onto a sufficiently large bounded domain, and assume periodic boundary conditions.
    
    \subsection{Examples in the classical regime}
    We first consider examples in 2D for the Dirac equation in the classical regime. Specifically,  we take $ \eps = \nu = \delta = 1, d = 2$ in (\ref{Dirac1d2d}) and the initial condition in (\ref{initial}) is taken as
    \begin{equation}
        \phi_1(0,\mathbf{x})=e^{-\frac{x^2+y^2}{2}},\quad \phi_2(0,\mathbf{x}) = e^{-\frac{(x-1)^2+y^2}{2}}, \qquad \mathbf{x} = (x,y)^T \in \mathbb{R}^2.
    \end{equation}
    We choose the honeycomb electric potential
    \begin{equation} \label{V}
        V(t,\mathbf{x}) = \cos\left(\frac{4\pi}{\sqrt{3}}\mathbf{e_1}(t)\cdot \mathbf{x}\right)+\cos\left(\frac{4\pi}{\sqrt{3}}\mathbf{e_2}(t)\cdot \mathbf{x}\right) + \cos\left(\frac{4\pi}{\sqrt{3}}\mathbf{e_3}(t)\cdot \mathbf{x}\right),
    \end{equation}
    with
    \begin{align}
        \mathbf{e_1}(t) &= \left(\cos(\theta(t)),\sin(\theta(t))\right)^T,\notag \\
        \mathbf{e_2}(t) &= \left(\cos(\theta(t)+\frac{2\pi}{3}),\sin(\theta(t)+\frac{2\pi}{3})\right)^T, \notag\\
        \mathbf{e_3}(t) &= \left(\cos(\theta(t)+\frac{4\pi}{3}),\sin(\theta(t)+\frac{4\pi}{3})\right)^T,
    \end{align}
    where $\theta(t)$ is a given function. In our examples, we consider $\theta(t)$ to be\\
    (1) $\theta(t) \equiv \pi$; \\
    (2) $\theta(t) = \pi+\pi t$;\\
    (3) $\theta(t) = \pi + \pi \cos(\pi t)$.\\
    
    In particular, the problem is solved numerically on a bounded domain $(-25, 25) \times (-25, 25)$. 
    Due to the fact that the exact solution is not available, we obtain a numerical `exact’ solution by using the $S_{6c}$ with a fine mesh size $h_e = 1/32 $ and a small time step $\tau_e = 10^{-4}$. Let $\Phi^n=(\Phi^n_0,\Phi^n_1,\cdots,\Phi_{M-1}^n,\Phi^n_M)^T$ be the numerical solution obtained by a numerical method with mesh size $h$ and time step $\tau$, the numerical errors for the wave function, probability density and current density are respectively quantified as 
    \begin{align*}
        e_\Phi(t_n) &= \Vert \Phi^n - \Phi(t_n,\cdot) \Vert_{l_2} := h\sqrt{ \sum_{j =0}^{M-1} \sum_{l =0}^{M-1} |\Phi(t_n,x_j,y_l)-\Phi_{jl}^n|^2},\\
        e_\rho(t_n) &= \Vert |\Phi^n|^2 - |\Phi(t_n,\cdot)|^2 \Vert_{l_2} := h\sqrt{ \sum_{j =0}^{M-1} \sum_{l =0}^{M-1} \big(|\Phi(t_n,x_j,y_l)|^2-|\Phi_{jl}^n|^2\big)^2},\\
        e_\mathbf{J}(t_n) &= \Vert \mathbf{J}(\Phi^n) - \mathbf{J}(\Phi(t_n,\cdot)) \Vert_{l_2} \\
        &:= h\sqrt{ \sum_{j =0}^{M-1} \sum_{l =0}^{M-1} \sum_{k=1}^{2}|(\Phi(t_n,x_j,y_l))^*\sigma_k \Phi(t_n,x_j,y_l) -(\Phi_{jl}^n)^*\sigma_k \Phi_{jl}^n |^2},
    \end{align*}
    where $M = 50/h,x_j := -25 + jh,y_l:=-25+lh$, and $\Phi_{jl}^n(j,l=0,\cdots,M,~n=0,1,\cdots,T/\tau)$ is the numerical solution at ($x_j,y_l$). The results are shown case by case.\\

    \renewcommand{\arraystretch}{1.2}
    \begin{table}[t]
        \centering
        \setlength{\tabcolsep}{15pt}
        \begin{tabular}{cccccc}
        \toprule
              &$h_0 = 1 $ & $h_0/2$ & $h_0/2^2$ & $h_0/2^3$ & $h_0/2^4$   \\
        \midrule
        $S_4$ &  2.62 & 1.10 &  1.01E-01 & 3.83E-04 & 7.33E-10\\
        $S_{\text{4RK}}$  & 2.62 & 1.10 & 1.01E-01& 3.83E-04 & 7.33E-10\\
        $S_{4\text{c}}$  & 2.62 & 1.10 & 1.01E-01 & 3.83E-04 & 7.34E-10\\
        $S_6$  & 2.62 & 1.10 & 1.01E-01 & 3.83E-04 & 7.34E-10\\
        $S_{6\text{c}}$ & 2.62 & 1.10 & 1.01E-01 & 3.83E-04 & 7.33E-10\\
        \bottomrule
        \end{tabular}
        \caption{Spatial errors $e_\Phi(t = 2)$ of different time-splitting methods under different mesh sizes $h$ for the Dirac equation (\ref{Dirac1d2d}) in 2D, with the potential given in (\ref{V}), where $\theta(t) \equiv \pi$.}
        \label{tab11}
    \end{table}
    
    \begin{table}[t]
        \centering
        \resizebox{\textwidth}{!}{%
        \begin{tabular}{ccccccccc}
        \hline
             & & \textbf{$\tau_0 = 1/2$} & \textbf{$\tau_0/2$} & \textbf{$\tau_0/2^2$} & \textbf{$\tau_0 /2^3$} & \textbf{$\tau_0/2^4$} & \textbf{$\tau_0/2^5$} & \textbf{$\tau_0/2^6$}  \\
        \hline
         $S_4$& $e_\Phi(t=2)$ &	1.09E+00&	6.45E-02&	8.85E-03&	7.09E-04&	4.72E-05&	3.00E-06&	1.88E-07 \\
    	& rate & - &	4.07& 	2.87& 	3.64& 	3.91& 	3.98& 	3.99 \\
    	&CPU time (s)&	0.54& 	0.88& 	1.61& 	3.09& 	6.13& 	12.27& 	$\mathbf{24.18}$ \\
        $S_{4\text{RK}}$ & $e_\Phi(t=2)$ &	2.09E-02&	8.93E-04&	1.86E-05&	1.12E-06&	6.88E-08&	4.29E-09&	2.67E-10 \\
    	& rate & - & 4.55&  5.58& 	4.06& 	4.02& 	4.00& 	4.00 \\ 
    	&CPU time (s)&	0.92& 	1.49& 	2.75& 	5.40& 	10.82& 	20.99& 	$\mathbf{42.10}$ \\
        $S_{4\text{c}}$ & $e_\Phi(t=2)$ &	1.69E-01&	7.98E-03&	1.98E-04&	1.18E-05&	7.28E-07&	4.54E-08&	2.84E-09 \\
    	& rate & - & 4.41& 	5.33& 	4.08& 	4.01& 	4.00& 	4.00 \\
    	&CPU time (s)&	0.38& 	0.57& 	1.02& 	2.05& 	3.74& 	7.42& 	$\mathbf{14.98}$ \\
        $S_6$ & $e_\Phi(t=2)$ &	3.70E-01&	1.47E-02&	3.10E-04&	5.32E-06&	8.51E-08&	1.34E-09&	3.25E-11 \\
    	& rate & - & 4.66 &	5.57 &	5.87 &	5.96 &	5.99 &	5.36 \\
    	&CPU time (s)& 1.15 &	1.99 &	3.73 &	7.30 &	14.07 & 27.65 &	$\mathbf{56.11}$\\
       $S_{6\text{c}}$ & $e_\Phi(t=2)$ & 3.83E-01&	1.52E-02&	3.24E-04&	5.70E-06&	9.19E-08&	1.45E-09&	3.41E-11 \\
    	& rate & - & 4.66& 	5.55& 	5.83& 	5.95& 	5.99 &	5.41 \\
    	&CPU time (s)&	0.59& 	1.01& 	1.91& 	3.63& 	7.13& 	14.04 &	$\mathbf{28.45}$ \\
        \hline
        \end{tabular}}
        \caption{Temporal errors $e_\Phi(t = 2)$ of different time-splitting methods under different time steps $\tau$ for the Dirac equation (\ref{Dirac1d2d}) in 2D, with the potential given in (\ref{V}), where $\theta(t) \equiv \pi$.}
        \label{tab12}
    \end{table}

    \begin{figure}[thbp]
        \centering
        \includegraphics[width=1.0\linewidth]{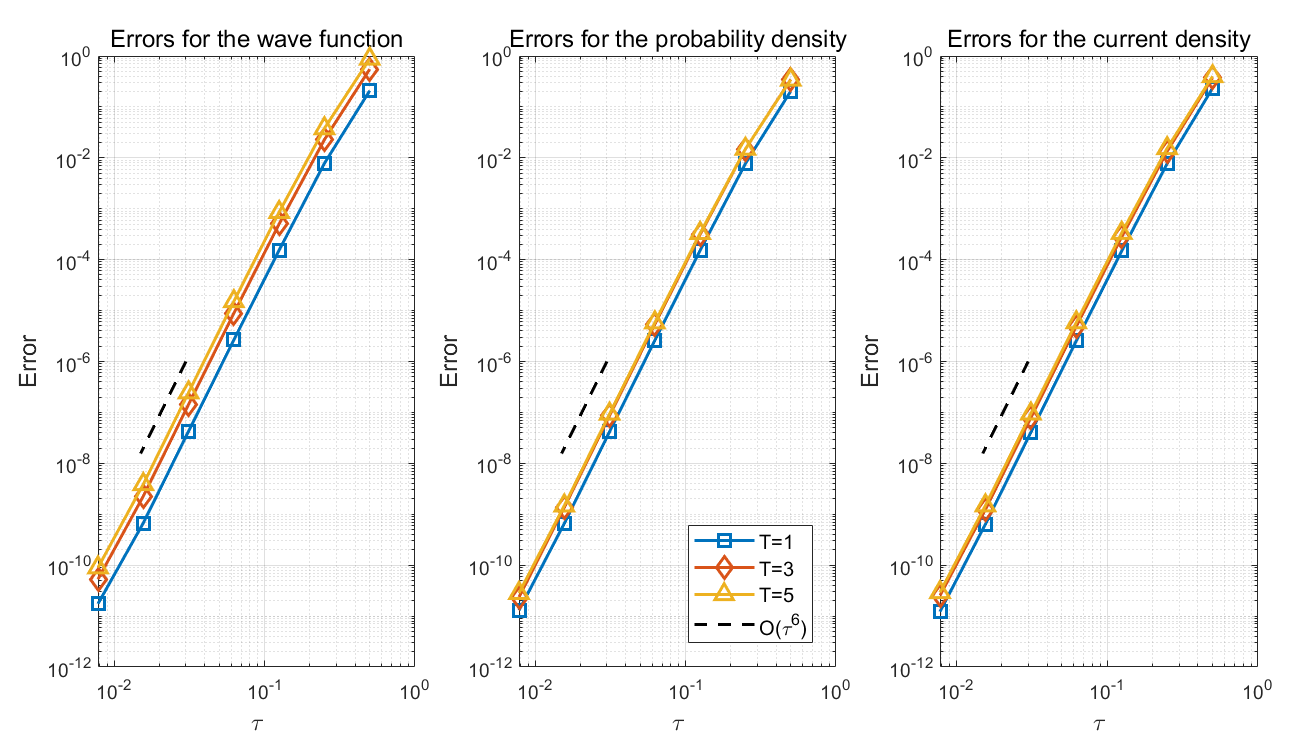}
        \caption{Temporal errors for the wave function, the probability density, and the current density with different $T$ for the Dirac equation (\ref{Dirac1d2d}) in 2D, with the potential given in (\ref{V}) where $\theta(t) \equiv \pi$.}
        \label{fig1}
    \end{figure}
    
    \noindent (1) $\theta(t) \equiv \pi$ 
    
    In this case, $\theta(t)$ is time-independent. In order to compare the spatial errors, we take the time step $\tau = \tau_e = 10^{-4}$ such that the temporal discretization error could be negligible. Table \ref{tab11} lists the numerical errors $e_\Phi(t = 2)$ for different time-splitting methods under different mesh sizes $h$. We observe that each row exhibits spectral convergence, confirming spatial accuracy of spectral order for all methods. In order to compare the temporal errors, we take the mesh size $h = h_e = 1 /32$ such that the spatial discretization error could be negligible. Table \ref{tab12} presents the numerical errors $e_\Phi(t = 2)$ for various time-splitting methods across different time steps $\tau$. 
    
    From Tables \ref{tab11} and \ref{tab12}, we have the following observations: (i) All five methods exhibit spectral accuracy in space. $S_4,S_{4\text{RK}}$ and $S_{4\text{c}}$ achieve fourth-order accuracy in time, while $S_6$ and $S_{6\text{c}}$ attain sixth-order accuracy. (ii) For any fixed mesh size $h$ and time step $\tau$: $S_4$ and $S_{4\text{RK}}$ require approximately 1.5 times and 3 times the computational time of $S_{4\text{c}}$, respectively. While $S_6$ takes roughly twice as long as $S_{6\text{c}}$. (iii) Among the three fourth-order time-splitting methods, their errors satisfy  $S_{4\text{RK}} < S_{4\text{c}} \ll S_4$. For the sixth-order methods, their errors are very close for a fixed $\tau$ and they are significantly smaller than the errors of the fourth-order methods, especially when $\tau$ is small. (iv) There is an obvious reduction in order for $S_4$ when $\tau$ is large, while the other methods maintain their respective orders. 
    
    Therefore, we can conclude that the sixth-order methods ($S_6$ and $S_{6\text{c}}$) outperform fourth-order schemes in precision. $S_{6\text{c}}$ offers a significant speedup over $S_{6}$, with execution times reduced by nearly a factor of two. This establishes $S_{6\text{c}}$ as an optimal method among all these high order splitting methods for the Dirac equations without magnetic potentials.
    
    We also exhibit the temporal errors $e_\Phi(T), e_\rho(T)$ and $e_\mathbf{J}(T)$ with different time steps for different $T$ in Figure \ref{fig1}. The results confirm that $S_{6\text{c}}$ achieves sixth-order temporal convergence for the wave function, the probability density and the current density.

    \begin{figure}[thbp]
        \centering
        \includegraphics[width=1.0\linewidth]{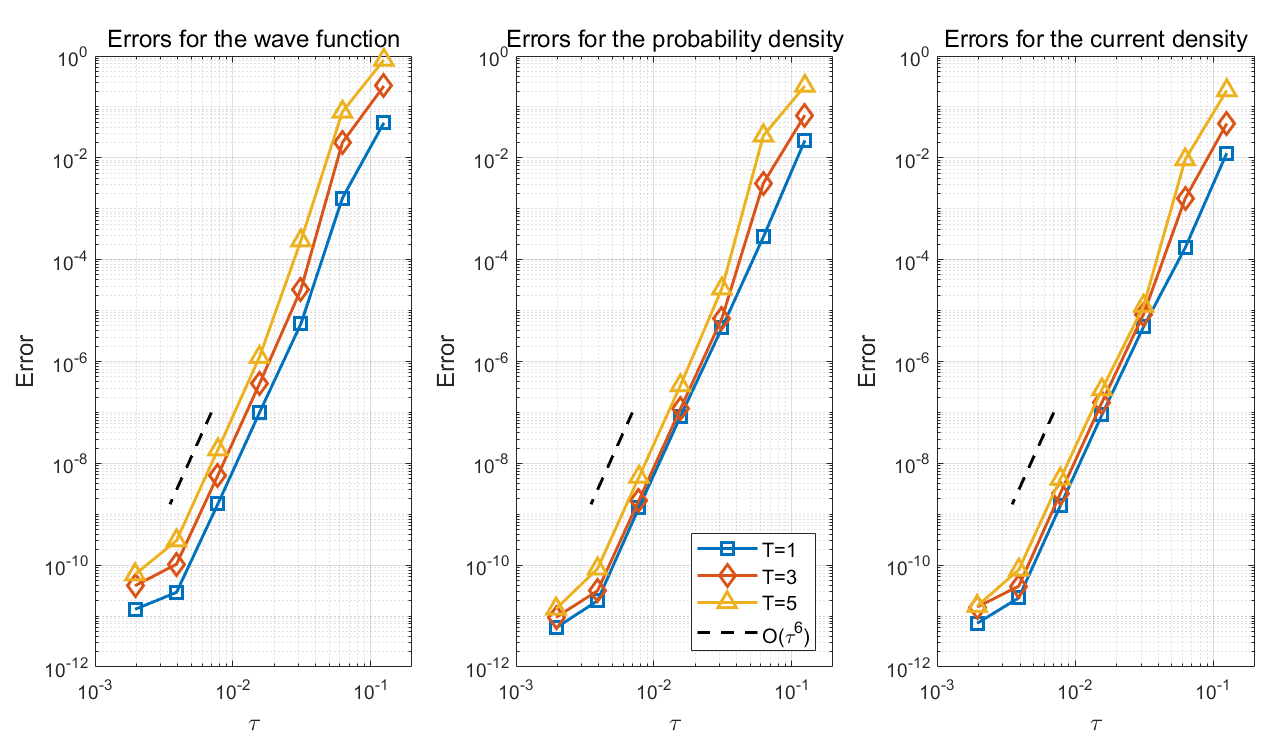}
        \caption{Temporal errors for the wave function, the probability density, and the current density with different $T$ for the Dirac equation (\ref{Dirac1d2d}) in 2D, with the potential given in (\ref{V}) where $\theta(t) = \pi + \pi t$. }
        \label{fig2}
    \end{figure}
    
    \begin{figure}[thbp]
        \centering
        \includegraphics[width=1.0\linewidth]{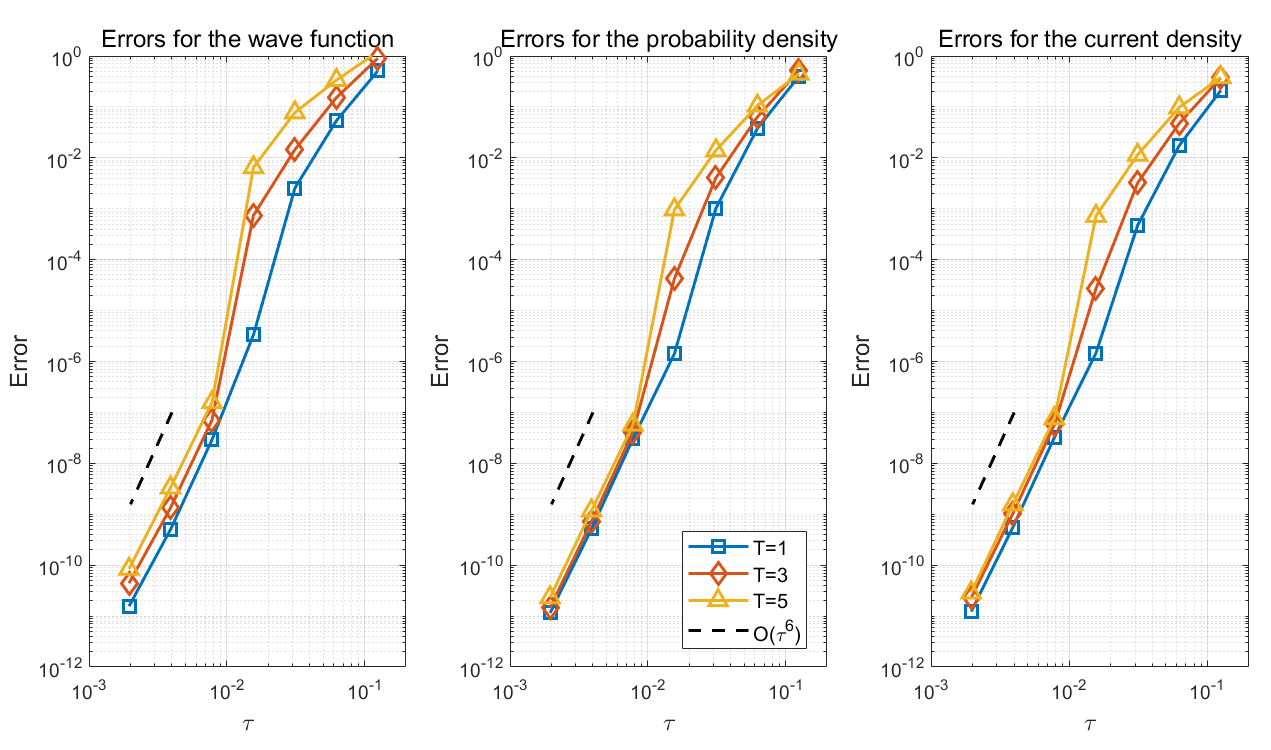}
        \caption{Temporal errors for the wave function, the probability density, and the current density with different $T$ for the Dirac equation (\ref{Dirac1d2d}) in 2D, with the potential given in (\ref{V}) where $\theta(t) = \pi + \pi \cos(\pi t)$.} 
        \label{fig3}
    \end{figure}
    
    \vspace{20pt}
    
    \noindent (2) $\theta(t) = \pi+\pi t$
    
    In this case, $\theta(t)$ is monotonically increasing, and $V(t) := V(t, \mathbf{x})$ represents a periodic electric potential function. We use $S_\text{6c}$ with time-dependent potentials developed in Section~\ref{sec:tdS6c} to propagate the wave function in time. Figure \ref{fig2} presents the three errors $e_\Phi(T)$, $e_\rho(T)$, and $e_{\mathbf{J}}(T)$ with different $T$ for this case. The figure shows that the solution exhibits sixth-order convergence in time for the wave function as well as the observables. These results confirm the efficacy of $S_{6\text{c}}$ with the time-ordering technique for solving the 2D Dirac equation with time-dependent potentials.
    
    \vspace{20pt}
    
    \noindent (3) $\theta(t) = \pi + \pi \cos(\pi t)$
    
    In this case, $\theta(t)$ is periodic in time, which generates a periodic electric potential $V(t):=V(t,\mathbf{x})$ with the same period. We still apply the $S_\text{6c}$ developed in Section~\ref{sec:tdS6c}. Figure \ref{fig3} presents the three errors $e_\Phi(T)$, $e_\rho(T)$, and $e_{\mathbf{J}}(T)$ with different $T$ for this case. The results are similar to those in case (2), which again validate the effectiveness of the $S_{6\text{c}}$ for solving the Dirac equation in 2D with time-dependent potentials.
    
    \subsection{An example in the nonrelativistic regime}
    
    In this example, we consider the Dirac equation in the nonrelativistic regime, characterized by $\delta = \nu = 1$ and $0 < \eps \ll 1$, where the speed of wave propagation is significantly slower than the speed of light. As established in \cite{Bao2020}, first- and second-order time-splitting methods exhibit super-resolution, i.e. they are uniformly convergent with respect to the small parameter $\eps$ when solving the Dirac equation without magnetic potentials in this regime. We will demonstrate through this example that $S_{6\text{c}}$ retains the super-resolution property.

    We consider the Dirac equation (\ref{Dirac1d2d}) with $d=1$ and $\delta = \nu = 1$.
    The initial condition in (\ref{initial}) is taken as
    \begin{equation}
        \phi_1(0,x) = e^{-x^2/2}, \qquad \phi_2(0,x) = e^{-(x-1)^2/2}, \qquad x\in \mathbb{R}.
    \end{equation}
    Consider the electric potential
    \begin{equation}
        V(x) = \frac{1-x}{1+x^2}, \qquad x\in \mathbb{R}.
    \end{equation}
    We also truncate the whole space onto a bounded domain $\Omega = (a, b)$ with negligible spatial error, and choose the mesh size $h:=\Delta x=\dfrac{b-a}{M}$ with $M$ an even positive integer. Then the grid points can be denoted as $x_j:=a + jh~(j=0,1,\cdots,M)$.
    
    To show the numerical results, we introduce the discrete $l_2$ errors of the numerical solution. Let $\Phi^n$ be the numerical solution obtained by a numerical method with time step $\tau$ and $\eps$ as well as a very fine mesh size $h$ at time $t = t_n$, and let $\Phi(t, x)$ be the exact solution. Then the relative discrete $l_2$ error is given by
    \begin{equation}
        e^{\eps,\tau}(t_n) = \Vert \Phi^n - \Phi(t_n,\cdot) \Vert_{l_2} = \sqrt{h \sum_{j =0}^{M-1} |\Phi(t_n,x_j)-\Phi_j^n|^2}.
    \end{equation}

     \begin{table}[t]
        \resizebox{\textwidth}{!}{
        \centering
        \begin{tabular}{cccccccc}
        \hline
             $e^{\eps,\tau}(t=2\pi)$ &  \textbf{$\tau_0 = \pi/4$} & \textbf{$\tau_0/4$} & \textbf{$\tau_0/4^2$} & \textbf{$\tau_0 /4^3$} & \textbf{$\tau_0/4^4$} & \textbf{$\tau_0/4^5$}  \\
        \hline
         $\eps_0=1/2$  & 1.32E-01 &	2.33E-04 &	6.68E-08 &	5.92E-10 &	5.91E-10 &	5.87E-10\\
         rate & - & 4.57  &	5.88 &	3.41 &	0.00 &	0.00 \\
         \hline
         $\eps_0/2$ & \textbf{3.23E-01} &	2.47E-02 &	1.08E-05 &	3.02E-09 &	6.86E-10 &	6.83E-10\\
         rate & - & 1.85 &	5.58 &	5.90 &	1.07 &	0.00 \\
         \hline
         $\eps_0/2^2$ & 1.79E-01 &	\textbf{1.56E-01} &	1.12E-02 &	6.11E-07 &	6.08E-10 &	5.72E-10\\
         rate & - & 0.10 &	1.90 &	7.08 &	4.99 &	0.04 \\
         \hline
         $\eps_0/2^3$ & 1.14E-01	& 7.47E-02 & \textbf{7.78E-02}	& 5.52E-03 &	3.72E-08 &	7.50E-10\\
         rate & - & 0.30  &	-0.03 &	1.91 &	8.59 &	2.82 \\
         \hline
         $\eps_0/2^4$ & 1.07E-01 &	3.04E-02 &	3.70E-02	& \textbf{3.89E-02} &	2.75E-03 &	2.43E-09\\
         rate & - & 0.91 &	-0.14 &	-0.04  &	1.91 &	10.06 \\
         \hline
         $\eps_0/2^5$ & 9.88E-02 &	2.26E-02 &	1.48E-02 &	1.85E-02 &	\textbf{1.94E-02} &	1.37E-03\\
         rate & - & 1.06 &	0.31 &	-0.16  & -0.03  &	1.91 \\
         \hline
         $\eps_0/2^6$ & 9.74E-02	& 1.09E-02 &	1.08E-02 & 7.36E-03 &	9.23E-03 & \textbf{9.72E-03}\\
         rate & - & 1.58 &	0.01 &	0.28 &	-0.16 &	-0.04 \\
         \hline
         $\eps_0/2^7$ & 9.72E-02 &	7.65E-03 &	4.33E-03 &	5.37E-03 &	3.68E-03 &	4.61E-03\\
         rate & - & 1.83 &	0.41 &	-0.16 &	0.27 &	-0.16 \\
         \hline
         $\eps_0/2^9$ & 9.69E-02 &	6.82E-03 &	1.27E-03 &	8.91E-04 &	1.08E-03 &	1.34E-03\\
         rate & - & 1.91  &	1.21  &	0.26 &	-0.14 &	-0.16 \\
         \hline
         \hline
         $\max_{0<\eps\le1} e^{\eps,\tau}(t=2\pi)$ & 3.23E-01	& 1.56E-01&	7.78E-02 &	3.89E-02 &	1.94E-02 &	9.72E-03\\
         rate & - & 0.52 & 0.50 & 0.50 &	0.50 & 	0.50 \\   
        \hline
        \end{tabular}}
        \caption{Discrete $l_2$ temporal errors $e^{\eps,\tau}(t = 2\pi)$ for the wave function with resonant time step size }
        \label{tab16}
    \end{table}
    
    \begin{table}[t]
        \resizebox{\textwidth}{!}{
        \centering
        \begin{tabular}{ccccccccc}
        \hline
             $e^{\eps,\tau}(t=4)$ &  \textbf{$\tau_0 = 1$} & \textbf{$\tau_0/4$} & \textbf{$\tau_0/4^2$} & \textbf{$\tau_0 /4^3$} & \textbf{$\tau_0/4^4$} & \textbf{$\tau_0/4^5$} & \textbf{$\tau_0/4^6$} \\
        \hline
         $\eps_0=1$  & 2.30E-02 & 9.32E-06 &	2.40E-09 &	2.92E-10 & 2.91E-10 &	2.90E-10 &	2.83E-10\\
         rate & - & 5.63 & 5.96 & 1.52 &	0.00 &	0.00 &	0.02 \\
         \hline
         $\eps_0/2$ & \textbf{2.39E-01} &	8.42E-04 &	2.66E-07 &	3.95E-10 &	3.88E-10 &	3.86E-10 &	3.78E-10\\
         rate & - & 4.07 &	5.81 &	4.70 &	0.01 &	0.00 &	0.02 \\
         \hline
         $\eps_0/2^2$ & 1.89E-01	& 1.69E-02	& 4.00E-05 & 	1.16E-08 &	3.75E-10 &	3.73E-10 &	3.65E-10\\
         rate & - & 1.74 &	4.36 &	5.88 &	2.48 & 	0.00 &	0.02 \\
         \hline
         $\eps_0/2^3$ & 1.36E-01 &	1.26E-02 &	1.46E-03 &	2.28E-06 &	7.41E-10 &	3.31E-10 &	3.26E-10\\
         rate & - & 1.72 &	1.55 &	4.66 &	5.79 &	0.58 &	0.01 \\
         \hline
         $\eps_0/2^4$ &1.54E-01 & 1.03E-02 &	1.65E-03 &	1.58E-04	& 1.39E-07 &	3.35E-10 &	3.21E-10\\
         rate & - & 1.95 &	1.32 &	1.69 &	5.08 &	4.35 &	0.06 \\
         \hline
         $\eps_0/2^5$ & 1.45E-01 &	1.22E-02 &	7.87E-04 &	2.23E-04 &	1.86E-05 &	8.82E-09 &	3.49E-10\\
         rate & - & 1.79 &	1.98 &	0.91 &	1.79 &	5.52 &	2.33 \\
         \hline
         $\eps_0/2^6$ & 1.39E-01 &	\textbf{2.26E-02} &	6.87E-04 &	7.50E-05 &	\textbf{2.82E-05} &	2.24E-06 &	7.21E-10\\
         rate & - & 1.31 &	2.52 &	1.60 &	0.71 &	1.83 &	5.80 \\
         \hline
         $\eps_0/2^7$ & 1.38E-01 &	1.11E-02 &	\textbf{4.19E-03} & 4.77E-05 &	8.93E-06 &	\textbf{3.63E-06} &	2.82E-07\\
         rate & - & 1.82 &	0.70 &	3.23 &	1.21 &	0.65 &	1.84 \\
         \hline
         $\eps_0/2^8$ & 1.38E-01 &	1.01E-02 &	6.85E-04 &	\textbf{2.36E-04} &	4.04E-06 &	1.10E-06 &	\textbf{4.59E-07}\\
         rate & - & 1.89 &	1.94 &	0.77 &	2.93 &	0.94 &	0.63 \\
         \hline
         $\eps_0/2^{10}$ & 1.38E-01 &	1.01E-02 &	6.36E-04 &	3.96E-05 &	3.25E-06 &	2.11E-06 &	4.84E-08\\
         rate & - & 1.89 &	1.99 & 	2.00 &	1.80  &	0.31 &	2.72 \\
         \hline
         \hline
         $\max_{0<\eps\le1} e^{\eps,\tau}(t=4)$ & 2.39E-01	& 2.26E-02 &	4.19E-03 &	2.36E-04 &	2.82E-05 &	3.63E-06 &	4.59E-07\\
         rate & - & 1.70 &	1.22 &	2.08 &	1.53 &	1.48 &	1.49\\
        \hline
        \end{tabular}}
        \caption{Discrete $l_2$ temporal errors $e^{\eps,\tau}(t = 4)$ for the wave function with nonresonant time step size }
        \label{tab17}
    \end{table}
    
    In this numerical experiment, we consider both resonant time steps (which are multiples of $\eps^2 \pi$ for some given $\eps$) and non-resonant time steps. For the resonant case, the computational domain is set to be $\Omega = (-32, 32)$, while for the non-resonant case, we use $\Omega = (-16, 16)$. In all the tests, we fix the mesh size $h = \frac{1}{16}$ to compare the temporal errors. Reference solutions are generated with $\tau_e = 2\pi \times 10^{-6}$ for the resonant case and $\tau_e = 8 \times 10^{-6}$ for the non-resonant case. Table \ref{tab16} presents numerical errors $e^{\eps,\tau}(t=2\pi)$ for the resonant case with different time steps $\tau$ and $\eps$, while Table \ref{tab17} shows errors $e^{\eps,\tau}(t=4)$ for the non-resonant case. In both tables, bold terms indicate the maximum error within each column.
    
    In Table \ref{tab16}, the last two rows give the largest error of each column, revealing at least $1/2$ order of uniform convergence under resonant time steps. Similarly, the last two rows of Table \ref{tab17} indicates at least $3/2$ order of uniform convergence under non-resonant time steps. These results verify that the developed $S_\text{6c}$ maintains the super-resolution property for solving the Dirac equation in the nonrelativistic regime. 
    
    \section{Conclusions and future work} \label{sec5}
    We proposed a novel sixth-order compact time-splitting Fourier pseudospectral method, denoted as $S_{6c}$, for solving the Dirac equation in the absence of magnetic potentials. The method is straightforward to implement. It achieves sixth-order accuracy in time and spectral precision in space. The key advantage lies in the substantial reduction in the number of required exponential operators compared to the existing high-order splitting schemes, resulting in a significant decrease in computational time. Extensive numerical tests in one- and two-dimensional spaces confirm the method's efficiency and accuracy for simulating the dynamics of the Dirac equations without magnetic potentials. Specifically, $S_{6\text{c}}$ exhibits significantly higher accuracy than the fourth-order schemes $S_4$ and $S_{4\text{c}}$, and it outperforms both the fourth-order $S_{4\text{RK}}$ and the sixth-order $S_6$ in computational efficiency. In addition, the example in the nonrelativistic regime shows that $S_{6\text{c}}$ maintains the super-resolution property of time-splitting methods.

    However, extending the ideas to construct higher-order methods, such as eighth-order compact splitting schemes, remains a challenge. This is due to the requirement of using more terms in the BCH formulas (\ref{BCH}) and (\ref{BCH2}), as well as computing higher-order commutators involving $W$ and $T$. As a result, a systematic study regarding the construction of higher-order compact time-splitting methods is left for future investigation. \\
    
\section*{Acknowledgements}
Z. He and J. Yin were partially supported by Shanghai Rising-Star Program under Grant No. 24QA2700600 and National Natural Science Foundation of China under Grant No. 12571422. 

\bibliographystyle{plain} 
\bibliography{references} 

\setcounter{equation}{0}  

\vspace{20pt}

\begin{center} \label{appA}
	{\bf Appendix A}. Computation of $V_2$
\end{center}
\setcounter{equation}{0}
\renewcommand{\theequation}{A.\arabic{equation}}

As mentioned before, we know
\begin{align*}
    V_1 =& \tau(2c_1T + c_0W) - \tau^3\frac{1}{6}[c_1T,c_1T,c_0W] \notag\\
    &+\tau^5 \left(\frac{7}{360}[c_1T,c_1T,c_1T,c_1T,c_0W] +\frac{1}{45}[c_0W,c_1T,c_1T,c_1T,c_0W] \right)+\mathcal{O}(\tau^7)
\end{align*}
Using Theorem \ref{th2} to expand
\begin{equation*}
    e^{\tau c_2 W}e^{V_1}e^{\tau c_2 W} = e^{V_2},
\end{equation*}
we have
\begin{align*} 
    V_2 =& \tau(2c_2W+2c_1T+c_0W)-\tau^3 \frac{1}{6}c_0c_1^2[T,T,W] 
    +\tau^5 \frac{7}{360}c_0c_1^4[T,T,T,T,W]  \\
    &+\tau^5 \frac{1}{45}c_0^2c_1^3[W,T,T,T,W]+ \frac{1}{6} [V_1,V_1,\tau c_2 W] - \frac{1}{6}[\tau c_2 W,\tau c_2 W,V_1]  \\
    &+ \frac{7}{360}[\tau c_2 W,\tau c_2 W,\tau c_2 W,\tau c_2 W,\tau(2c_1T+c_0W)] \\
    &- \frac{1}{360}[\tau(2c_1T+c_0W),\tau(2c_1T+c_0W),\tau(2c_1T+c_0W),\tau(2c_1T+c_0W),\tau c_2W] \\
    &+ \frac{1}{90} [\tau c_2 W,\tau(2c_1T+c_0W),\tau(2c_1T+c_0W),\tau(2c_1T+c_0W),\tau c_2 W] \\
    &+ \frac{1}{45}[\tau(2c_1T+c_0W),\tau c_2 W,\tau c_2 W,\tau c_2 W,\tau(2c_1T+c_0W)] \\
    &-\frac{1}{60}[\tau c_2 W,\tau c_2 W,\tau(2c_1T+c_0W),\tau(2c_1T+c_0W),\tau c_2 W] \\
    &+ \frac{1}{30}[\tau(2c_1T+c_0W),\tau(2c_1T+c_0W),\tau c_2 W\tau c_2 W,\tau(2c_1T+c_0W)] + \mathcal{O}(\tau^7) 
\end{align*}
Noticing that $[*,W,W] = 0,[*,T,T]=0$~($*$ is one or more operators) and according to Lemma \ref{lemma4}, removing the vanishing terms, $V_2$ simplifies to:
\begin{align}  \label{A.1}
    V_2 =& \tau(2c_2W+2c_1T+c_0W)-\tau^3 \frac{1}{6}c_0c_1^2[T,T,W] 
    +\tau^5 \frac{7}{360}c_0c_1^4[T,T,T,T,W]   \notag\\
    &+\tau^5 \frac{1}{45}c_0^2c_1^3[W,T,T,T,W]+ \frac{1}{6} [V_1,V_1,\tau c_2 W] \notag \\
    &- \frac{1}{360}[\tau(2c_1T+c_0W),\tau2c_1T,\tau2c_1T,\tau2c_1T,\tau c_2W] \notag \\
    &+ \frac{1}{90} [\tau c_2 W,\tau2c_1T,\tau2c_1T,\tau2c_1T,\tau c_2 W]  + \mathcal{O}(\tau^7)
\end{align}
where
\begin{align}
    [ V_1,V_1,\tau c_2 W] =& [\tau(2c_1T + c_0W) - \tau^3\frac{1}{6}[c_1T,c_1T,c_0W], \notag\\
    &\tau(2c_1T + c_0W) - \tau^3\frac{1}{6}[c_1T,c_1T,c_0W],\tau c_2 W ] + \mathcal{O}(\tau^7) 
\end{align}
Based on Lemma \ref{lemma3} and Lemma \ref{lemma4}, it can be observed that
\begin{align*}
    [T,[T,T,W],W] &= -[T,W,[T,T,W]] = -[T,W,T,T,W] = 0; \\
    [W,[T,T,W],W] &= -[W,W,[T,T,W]] = -[W,W,T,T,W] = 0.
\end{align*}
and $[*,W,W] = 0$($*$ is one or more operators), thus
\begin{align}
    [ V_1,V_1,\tau c_2 W] =& [\tau(2c_1T + c_0W) - \tau^3\frac{1}{6}[c_1T,c_1T,c_0W],\tau 2c_1T,\tau c_2 W] + \mathcal{O}(\tau^7)\notag\\
    =& [\tau(2c_1T + c_0W),\tau 2c_1T,\tau c_2 W]\notag\\
    &-[ \tau^3\frac{1}{6}[c_1T,c_1T,c_0W],\tau 2c_1T,\tau c_2 W] + \mathcal{O}(\tau^7)\notag\\
    =& \tau^3 4c_1^2c_2[T,T,W] + \tau^5 (-\frac{1}{3})c_0c_1^3c_2[[T,T,W],T,W] + \mathcal{O}(\tau^7)
\end{align}
Again based on Lemma \ref{lemma3} and Lemma \ref{lemma4}, we know
\begin{equation*}
    [[T,T,W],T,W]+[T,W,[T,T,W]]+[W,[T,T,W],T] = 0,
\end{equation*}
and
\begin{equation*}
    [T,W,[T,T,W]] = [T,W,T,T,W] =0,
\end{equation*}
thus
\begin{align*}
    [W,[T,T,W],T] = -[W,T,[T,T,W]] = -[W,T,T,T,W],
\end{align*}
therefore 
\begin{equation}  \label{A.4}
    [[T,T,W],T,W] = [W,T,T,T,W].
\end{equation}
It is clear that $[T,T,W]=-[T,W,T]$, then
\begin{equation} \label{A.5}
    [ V_1,V_1,\tau c_2 W] = \tau^3 (-4)c_1^2c_2[T,W,T] + \tau^5 (-\frac{1}{3})c_0c_1^3c_2[W,T,T,T,W] + \mathcal{O}(\tau^7)
\end{equation}
Plugging (\ref{A.5}) into (\ref{A.1}), this reduces to:
\begin{align}  \label{A.6}
    V_2 =& \tau\left(2c_1 T + (c_0 + 2c_2)W \right) +\tau^3 \left(\frac{1}{6} c_1^2(c_0-4c_2)[T,W,T] \right) \notag\\
    &+\tau^5 (\frac{1}{45}c_0^2c_1^3-\frac{1}{18}c_0c_1^3c_2-\frac{1}{45}c_0c_1^3c_2+\frac{4}{45}c_1^3c_2^2)[W,T,T,T,W] \notag \\
    &+\tau^5 (\frac{7}{360}c_0c_1^4-\frac{2}{45}c_1^4c_2)[T,T,T,T,W] + \mathcal{O}(\tau^7)
\end{align}
which is exactly (\ref{V2}).

\vspace{80pt}

\begin{center} \label{appB}
	{\bf Appendix B}. Computation of $V_3$
\end{center}
\setcounter{equation}{0}

\renewcommand{\theequation}{B.\arabic{equation}}

Applying Theorem \ref{th2} to 
\begin{equation*}
    e^{\tau c_3 T}e^{V_2}e^{\tau c_3 T} =e^{V_3}
\end{equation*}
we obtain
\begin{align*}
    V_3 =&  \tau(2(c_3+c_1)T+(c_0+2c_2)W)+\tau^3 \left(\frac{1}{6} c_1^2(c_0-4c_2)[T,W,T] \right) \\
    &+\tau^5(\frac{1}{45}c_0^2c_1^3-\frac{1}{18}c_0c_1^3c_2-\frac{1}{45}c_0c_1^3c_2+\frac{4}{45}c_1^3c_2^2)[W,T,T,T,W]  \\
    &+\tau^5(\frac{7}{360}c_0c_1^4-\frac{2}{45}c_1^4c_2)[T,T,T,T,W] + \frac{1}{6} [V_2,V_2,\tau c_3 T] - \frac{1}{6}[\tau c_3 T,\tau c_3 T,V_2] \\
    & + \frac{7}{360}[\tau c_3 T,\tau c_3 T,\tau c_3 T,\tau c_3 T,\tau(2c_1T+(c_0+2c_2)W)] \\
    &- \frac{1}{360}[\tau(2c_1T+(c_0+2c_2)W),\tau(2c_1T+(c_0+2c_2)W),...\\
    &~~~~~~~~~~~~\tau(2c_1T+(c_0+2c_2)W),\tau(2c_1T+(c_0+2c_2)W),\tau c_3T] \\
    &+ \frac{1}{90} [\tau c_3T,\tau(2c_1T+(c_0+2c_2)W),\tau(2c_1T+(c_0+2c_2)W),...\\
    &~~~~~~~~~~\tau(2c_1T+(c_0+2c_2)W),\tau c_3T] \\
    &+ \frac{1}{45}[\tau(2c_1T+(c_0+2c_2)W),\tau c_3T,\tau c_3T,\tau c_3T,\tau(2c_1T+(c_0+2c_2)W)] \\
    &-\frac{1}{60}[\tau c_3T,\tau c_3T,\tau(2c_1T+(c_0+2c_2)W)),\tau(2c_1T+(c_0+2c_2)W),\tau c_3T] \\
    &+ \frac{1}{30}[\tau(2c_1T+(c_0+2c_2)W),\tau(2c_1T+(c_0+2c_2)W),...\\
    &~~~~~~~~~~~\tau c_3T,\tau c_3T,\tau(2c_1T+(c_0+2c_2)W)] + \mathcal{O}(\tau^7) 
\end{align*}

Noticing that $[*,W,W] = 0,[*,T,T]=0$~($*$ is one or more operators) and applying Lemma \ref{lemma4}, $V_3$ can be reformulated by excluding null entries:
\begin{align}  \label{B.1}
    V_3 =&  \tau(2(c_3+c_1)T+(c_0+2c_2)W)+\tau^3 \left(\frac{1}{6} c_1^2(c_0-4c_2)[T,W,T] \right)\notag\\
    &+\tau^5 (\frac{7}{360}c_0c_1^4-\frac{2}{45}c_1^4c_2)[T,T,T,T,W] \notag \\
    &+\tau^5(\frac{1}{45}c_0^2c_1^3-\frac{1}{18}c_0c_1^3c_2
    - \frac{1}{45}c_0c_1^3c_2+\frac{4}{45}c_1^3c_2^2)[W,T,T,T,W]  \notag\\
    &+ \frac{1}{6} [V_2,V_2,\tau c_3 T] - \frac{1}{6}[\tau c_3 T,\tau c_3 T,V_2] \notag\\
    &+ \frac{7}{360}[\tau c_3 T,\tau c_3 T,\tau c_3 T,\tau c_3 T,\tau(c_0+2c_2)W] \notag\\
    &- \frac{1}{360}[\tau(2c_1T+(c_0+2c_2)W),\tau 2c_1T,\tau 2c_1T,\tau(c_0+2c_2)W,\tau c_3T] \notag\\
    &+ \frac{1}{90} [\tau c_3T,\tau 2c_1T,\tau 2c_1T,\tau(c_0+2c_2)W,\tau c_3T] \notag\\
    &+ \frac{1}{45}[\tau(2c_1T+(c_0+2c_2)W),\tau c_3T,\tau c_3T,\tau c_3T,\tau(c_0+2c_2)W] \notag \\
    &-\frac{1}{60}[\tau c_3T,\tau c_3T,\tau 2c_1T,\tau(c_0+2c_2)W,\tau c_3T] \notag\\
    &+ \frac{1}{30}[\tau(2c_1T+(c_0+2c_2)W),\tau 2c_1T,\tau c_3T,\tau c_3T,\tau (c_0+2c_2)W] + \mathcal{O}(\tau^7) 
\end{align}
where
\begin{align}
    [V_2,V_2,\tau c_3 T]
    =& [\tau\left(2c_1 T + (c_0 + 2c_2)W \right) +\tau^3 \left(\frac{1}{6} c_1^2(c_0-4c_2)[T,W,T] \right),... \notag\\
    & ~~~\tau\left(2c_1 T + (c_0 + 2c_2)W \right) +\tau^3 \left(\frac{1}{6} c_1^2(c_0-4c_2)[T,W,T] \right),\tau c_3 T] \notag\\
    &+\mathcal{O}(\tau^7)  \notag\\
    =& [\tau\left(2c_1 T + (c_0 + 2c_2)W \right) +\tau^3 \left(\frac{1}{6} c_1^2(c_0-4c_2)[T,W,T] \right),... \notag\\
    & ~~~ \tau (c_0 + 2c_2)W  +\tau^3 \left(\frac{1}{6} c_1^2(c_0-4c_2)[T,W,T] \right),\tau c_3 T]+\mathcal{O}(\tau^7) \notag\\
    =& [\tau\left(2c_1 T + (c_0 + 2c_2)W \right),\tau (c_0 + 2c_2)W,\tau c_3 T] \notag\\
    &+ [\tau\left(2c_1 T + (c_0 + 2c_2)W \right),\tau^3 \left(\frac{1}{6} c_1^2(c_0-4c_2)[T,W,T] \right),\tau c_3 T]\notag\\
    &+ [\tau^3 \left(\frac{1}{6} c_1^2(c_0-4c_2)[T,W,T] \right),\tau (c_0 + 2c_2)W,\tau c_3 T] + \mathcal{O}(\tau^7)\notag\\
    =& \tau^3 2c_1c_3(c_0+2c_2)[T,W,T] + \tau^5 \Big( 
    \frac{1}{3}c_1^3c_3(c_0-4c_2))[T,[T,W,T],T] \notag\\
    &+ \frac{1}{6}c_1^2c_3(c_0+2c_2)(c_0-4c_2)[W,[T,W,T],T] \notag\\
    &+ \frac{1}{6}c_1^2c_3(c_0-4c_2)(c_0+2c_2)[[T,W,T],W,T]\Big) + \mathcal{O}(\tau^7)
\end{align}

By the properties of commutators and (\ref{A.4}), we establish:
\begin{align*}
    [T,[T,W,T],T] &= -[T,T,[T,W,T]] = [T,T,T,T,W], \\
    [W,[T,W,T],T] &= -[W,T,[T,W,T]] = [W,T,T,T,W], \\
    [[T,W,T],W,T] &= -[[T,T,W],W,T] = [[T,T,W],T,W] = [W,T,T,T,W].
\end{align*}
Thus
\begin{align}   \label{B.3}
    [V_2,V_2,\tau c_3 T] =& \tau^3 2c_1c_3(c_0+2c_2)[T,W,T] + \tau^5 \Big( 
    \frac{1}{3}c_1^3c_3(c_0-4c_2))[T,T,T,T,W] \notag\\
    &+ \frac{1}{6}c_1^2c_3(c_0+2c_2)(c_0-4c_2)[W,T,T,T,W]\notag\\
    &+ \frac{1}{6}c_1^2c_3(c_0-4c_2)(c_0+2c_2)[W,T,T,T,W]
    \Big)  + \mathcal{O}(\tau^7) \notag\\
    =& \tau^3 2c_1c_3(c_0+2c_2)[T,W,T] + \tau^5 \Big( 
    \frac{1}{3}c_1^3c_3(c_0-4c_2))[T,T,T,T,W]  \notag\\
    &+ \frac{1}{3}c_1^2c_3(c_0+2c_2)(c_0-4c_2)[W,T,T,T,W] \Big) + \mathcal{O}(\tau^7)
\end{align}

\begin{align} \label{B.4}
    [\tau c_3 T,\tau c_3 T,V_2] =& [\tau c_3 T,\tau c_3 T,\tau\left(2c_1 T + (c_0 + 2c_2)W \right)\notag\\ 
    &+\tau^3 \left(\frac{1}{6} c_1^2(c_0-4c_2)[T,W,T] \right)]+\mathcal{O}(\tau^7) \notag\\
    =& [\tau c_3 T,\tau c_3 T,\tau\left(2c_1 T + (c_0 + 2c_2)W \right)] \notag\\
    &+ [\tau c_3 T,\tau c_3 T,\tau^3 \left(\frac{1}{6} c_1^2(c_0-4c_2)[T,W,T] \right)] + \mathcal{O}(\tau^7) \notag\\
    =& \tau^3(c_0+2c_2)c_3^2[T,T,W] \notag\\
    &+ \tau^5 \frac{1}{6} c_1^2c_3^2(c_0-4c_2)[T,T,[T,W,T]] + \mathcal{O}(\tau^7) \notag\\
    =& -\tau^3(c_0+2c_2)c_3^2[T,W,T] \notag\\
    &+ \tau^5 (-\frac{1}{6})c_1^2c_3^2(c_0-4c_2)[T,T,T,T,W] + \mathcal{O}(\tau^7) 
\end{align}

Plugging (\ref{B.3}) and (\ref{B.4}) into (\ref{B.1}), this reduces to:
\begin{align} \label{B.5}
    V_3 =& \tau \left(2(c_1+c_3)T+(c_0+2c_2)W \right) \notag\\
         &+ \tau^3 \left( \frac{1}{6}c_1^2(c_0-4c_2)+\frac{1}{3}c_1c_3(c_0+2c_2)+\frac{1}{6}c_3^2(c_0+2c_2) \right)[T,W,T] \notag\\
         &+ \tau^5 \Big( c_1^4(\frac{7}{360}c_0-\frac{2}{45}c_2)+\frac{c_1^3c_3(c_0-4c_2)}{18}+\frac{c_1^2c_3^2(c_0-4c_2)}{36}+\frac{c_1^3c_3(c_0+2c_2)}{45}\notag\\
         &~~~~ +\frac{4}{45}c_1^2c_3^2(c_0+2c_2)+\frac{7}{90}c_1c_3^3(c_0+2c_2)+\frac{7}{360}c_3^4(c_0+2c_2)         
         \Big)[T,T,T,T,W] \notag\\
         &+ \tau^5 \Big( c_1^3(\frac{1}{45}c_0^2-\frac{7}{90}c_0c_2+\frac{4}{45}c_2^2)+\frac{1}{18}c_1^2c_3(c_0+2c_2)(c_0-4c_2) \notag\\
         &~~~~+\frac{1}{90}c_1^2c_3(c_0+2c_2)^2+\frac{1}{45}c_3^3(c_0+2c_2^2)\notag\\
         &~~~~+\frac{1}{15}c_1c_3^2(c_0+2c_2)^2
         \Big)[W,T,T,T,W] + \mathcal{O}(\tau^7)
\end{align}

\vspace{20pt}

\begin{center} \label{appC}
	{\bf Appendix C}. Computation of $V_4$
\end{center}
\setcounter{equation}{0}
\renewcommand{\theequation}{C.\arabic{equation}}

Applying Theorem \ref{th2} to 
\begin{equation*}
    e^{\tau c_4 W}e^{V_3}e^{\tau c_4 W} =e^{V_4}
\end{equation*}
we have
\begin{align*}
    V_4 =& \tau \left(2(c_1+c_3)T+(c_0+2c_2+2c_4)W \right) + \tau^3 \alpha_3[T,W,T] \\
    &+ \tau^5 (\alpha_4[T,T,T,T,W] + \alpha_5[W,T,T,T,W]) \\
    &+ \frac{1}{6} [V_3,V_3,\tau c_4 W] - \frac{1}{6}[\tau c_4 W,\tau c_4 W,V_3] \\
    &+ \frac{7}{360}[\tau c_4 W,\tau c_4 W,\tau c_4 W,\tau c_4 W,\tau \left(2(c_1+c_3)T+(c_0+2c_2)W \right)] \\
    &- \frac{1}{360}[\tau \left(2(c_1+c_3)T+(c_0+2c_2)W \right),\tau \left(2(c_1+c_3)T+(c_0+2c_2)W \right),...\\
    &~~~~~~~~~~~\tau \left(2(c_1+c_3)T+(c_0+2c_2)W \right),\tau \left(2(c_1+c_3)T+(c_0+2c_2)W \right),\tau c_4W] \\
    &+ \frac{1}{90} [\tau c_4 W,\tau \left(2(c_1+c_3)T+(c_0+2c_2)W \right),\tau \left(2(c_1+c_3)T+(c_0+2c_2)W \right),...\\
    &~~~~~~~~~~\tau \left(2(c_1+c_3)T+(c_0+2c_2)W \right),\tau c_4 W] \\
    &+ \frac{1}{45}[\tau \left(2(c_1+c_3)T+(c_0+2c_2)W \right),\tau c_4 W,\tau c_4 W,...\\
    &~~~~~~~~~~\tau c_4 W,\tau \left(2(c_1+c_3)T+(c_0+2c_2)W \right)] \\
    &-\frac{1}{60}[\tau c_4 W,\tau c_4 W,\tau \left(2(c_1+c_3)T+(c_0+2c_2)W \right),...\\
    &~~~~~~~~~~\tau \left(2(c_1+c_3)T+(c_0+2c_2)W \right),\tau c_4 W] \\
    &+ \frac{1}{30}[\tau \left(2(c_1+c_3)T+(c_0+2c_2)W \right),\tau \left(2(c_1+c_3)T+(c_0+2c_2)W \right),...\\
    &~~~~~~~~~~\tau c_4 W,\tau c_4 W,tau \left(2(c_1+c_3)T+(c_0+2c_2)W \right)] + \mathcal{O}(\tau^7) 
\end{align*}
Noticing that $[*,W,W] = 0,[*,T,T]=0$~($*$ is one or more operators) and applying Lemma \ref{lemma4}, by discarding the vanishing terms and reorganizing $V_4$, we obtain
\begin{align} \label{C.1}
    V_4 =& \tau \left(2(c_1+c_3)T+(c_0+2c_2+2c_4)W \right) + \tau^3 \alpha_3[T,W,T]\notag\\
    &+ \tau^5 (\alpha_4[T,T,T,T,W] + \alpha_5[W,T,T,T,W])+ \frac{1}{6} [V_3,V_3,\tau c_4 W] \notag\\
    & - \frac{1}{360}[\tau \left(2(c_1+c_3)T+(c_0+2c_2)W \right),\tau 2(c_1+c_3)T ,\tau 2(c_1+c_3)T,...\notag\\
    &~~~~~~~~~~~~\tau 2(c_1+c_3)T,\tau c_4W] \notag\\
    &+ \frac{1}{90} [\tau c_4 W,\tau 2(c_1+c_3)T,\tau 2(c_1+c_3)T,\tau 2(c_1+c_3)T,\tau c_4 W]  + \mathcal{O}(\tau^7)
\end{align}
where
\begin{align} \label{C.2}
    [V_3,V_3,\tau c_4 W] =& [\tau 2(c_1+c_3)T,\tau 2(c_1+c_3)T,\tau c_4 W] \notag\\
    &+ [\tau^3 \alpha_3[T,W,T],\tau 2(c_1+c_3)T,\tau c_4 W] \notag\\
    =& \tau^3 (-4)(c_1+c_3)^2c_4[T,W,T] + \tau^5 (-2)\alpha_3 (c_1+c_3)c_4[W,T,T,T,W]
\end{align}
The computation of (\ref{C.2}) here is entirely analogous to that of $[V_1, V_1, \tau c_2 W]$ in Appendix A. Substituting (\ref{C.2}) into (\ref{C.1}) and simplifying, we obtain
\begin{align}
    V_4 =& \tau(2(c_1+c_3)T+(c_0+2c_2+2c_4)W) + \tau^3\beta_3 [T,W,T] \notag\\
    &+\tau^5(\beta_4[T,T,T,T,W]+\beta_5[W,T,T,T,W])
    +\mathcal{O}(\tau^7),
\end{align}
where $\beta_3,\beta_4,\beta_5$ are shown in Table \ref{tab12}.

\end{document}